\documentclass[12pt]{amsart}

\usepackage{amssymb, enumitem, pbox} 
\usepackage[breaklinks=true,colorlinks=true,linkcolor=green,citecolor=red,urlcolor=blue,pdfencoding=auto,psdextra]{hyperref}

\renewcommand \a{\alpha}

\newcommand \la{\lambda}
\newcommand \ve{\varepsilon}

\newcommand \id{\mathrm{id}}
\newcommand \br{\mathbb{R}}
\newcommand \bc{\mathbb{C}}
\newcommand \bH{\mathbb{H}}
\newcommand \Oc{\mathbb{O}}

\newcommand \rk{\operatorname{rk}}

\newcommand \End{\operatorname{End}}

\renewcommand \Im{\operatorname{Im}}

\newcommand \Span{\operatorname{Span}}
\newcommand \Tr{\operatorname{Tr}}

\newcommand \Lin{\operatorname{Lin}}

\newcommand \mU{\mathcal{U}}

\newcommand\ag{\mathfrak a}
\newcommand\g{\mathfrak g}
\newcommand\cs{\mathfrak c}
\newcommand\h{\mathfrak h}
\newcommand\p{\mathfrak p}
\newcommand\z{\mathfrak z}
\newcommand\m{\mathfrak m}
\newcommand\q{\mathfrak q}
\renewcommand\k{\mathfrak k}
\newcommand \so{\mathfrak{so}}

\newcommand \ug{\mathfrak{u}}
\newcommand \su{\mathfrak{su}}
\renewcommand \sp{\mathfrak{sp}}
\newcommand \s{\mathfrak{s}}
\newcommand \n{\mathfrak{n}}
\renewcommand\t{\mathfrak t}
\newcommand \eg{\mathfrak{e}}
\newcommand \fg{\mathfrak{f}}

\newcommand \ad{\operatorname{ad}}
\newcommand \Ad{\operatorname{Ad}}
\newcommand \diag{\operatorname{diag}}

\newcommand \<{\langle}
\renewcommand \>{\rangle}
\newcommand \ip{\<\cdot,\cdot\>}
\newcommand \ic {\mathbf{i}}

\newcommand \SO{\mathrm{SO}}
\newcommand \SU{\mathrm{SU}}
\newcommand \Sp{\mathrm{Sp}}
\newcommand \Spin{\mathrm{Spin}}
\newcommand \E{\mathrm{E}}
\newcommand \F{\mathrm{F}}
\renewcommand \G{\mathrm{G}}

\theoremstyle{plain}
\newtheorem{theorem}{Theorem}
\newtheorem*{theorem*}{Theorem}

\newtheorem*{corollary*}{Corollary}
\newtheorem*{conj*}{Conjecture}
\newtheorem{lemma}{Lemma}
\newtheorem{proposition}{Proposition}
\newtheorem*{prop*}{Proposition}

\theoremstyle{definition}

\newtheorem*{definition*}{Definition}

\theoremstyle{remark}
\newtheorem{remark}{Remark}

\begin{document}

\title{Compact geodesic orbit spaces \\ with a simple isotropy group}

\author{Z.~Chen}
\address{School of Mathematics and Statistics, Guangdong University of Technology, Guangdong 510520, P.R. China}
\email{chenzhiqi@nankai.edu.cn}

\author{Y.~Nikolayevsky}
\address{Department of Mathematics and Physics, La Trobe University, Melbourne, Australia 3086}
\email{Y.Nikolayevsky@latrobe.edu.au}

\author{Yu.~Nikonorov}
\address{Southern Mathematical Institute of Vladikavkaz Scientific Centre of the Russian Academy of Sciences,
Vatutin str., 53, Vladikavkaz, Russia 362025}
\email{nikonorov2006@mail.ru}





\thanks{Supported by National Natural Science Foundation of China (11931009) and Natural Science Foundation of Tianjin (19JCYBJC30600). The second and the third author would like to thank S.S.Chern Institute of Mathematics and Nankai University (Tianjin, China) for their support and hospitality. The second author was partially supported by the Australian Research Council Discovery Grant DP210100951.} 

\begin{abstract}
Let $M=G/H$ be a compact, simply connected, Riemannian homogeneous space, where $G$ is (almost) effective and $H$ is a \emph{simple} Lie group. In this paper, we first classify all $G$-naturally reductive metrics on $M$, and then all $G$-geodesic orbit metrics on $M$.
\end{abstract}

\maketitle

\section{Introduction}
\label{s:intro}

A Riemannian manifold $(M,g)$ is called a \emph{geodesic orbit manifold} (or a manifold with homogeneous geodesics, or a GO manifold) if any geodesic of $M$ is an orbit of a $1$-parameter subgroup of the full isometry group of $(M,g)$ (without loss of generality, one can replace the full isometry group by its connected identity component). A Riemannian manifold $(M=G/H,g)$, where $H$ is a compact subgroup of the Lie group $G$ and $g$ is a $G$-invariant Riemannian metric on $M$, is called a \emph{$G$-geodesic orbit space} (or a space with $G$-homogeneous geodesics, or a $G$-GO space) if any geodesic of $M$ is an orbit of a $1$-parameter subgroup of the group $G$. Hence a Riemannian manifold $(M,g)$ is a geodesic orbit manifold, if it is a geodesic orbit space with respect to its full isometry group. This terminology was introduced in \cite{KV} by O.~Kowalski and L.~Vanhecke who initiated the systematic study of such spaces.

The GO property which plays the central role in this paper is a very general geometric phenomenon: it is extensively studied in Riemannian, Lorentzian and general pseudo-Riemannian settings, in Finsler geometry (see recent papers \cite{Du2, XDY, YD} and bibliographies therein), in affine geometry \cite{Du1}, and even for finite metric spaces \cite{BN4}. In all these cases, is not hard to see that the GO property implies homogeneity, but is much stronger.

The class of (Riemannian) geodesic orbit spaces includes (but is not limited to) symmetric spaces, weakly symmetric spaces \cite{BKV, Wo, Zil96}, normal and generalised normal homogeneous spaces,  naturally reductive spaces \cite{DZ}, Clifford-Wolf homogeneous manifolds \cite{BN2} and $\delta$-homogeneous manifolds \cite{BN1}. For the current state of knowledge in the theory of geodesic orbit spaces and manifolds we refer the reader to the book \cite{BN3} and the papers \cite{AN, Go, Nik2017} and the bibliographies therein.

Let $(M=G/H, g)$ be a homogeneous Riemannian space and let $\g=\h\oplus \p$ be an $\Ad(H)$-invariant decomposition, where $\g$ is the Lie algebra of $G$, $\h$ is the Lie algebra of $H$ and $\p$ is identified with the tangent space of $M$ at $eH$. The Riemannian metric $g$ is $G$-invariant and is determined by an $\Ad(H)$-invariant inner product $(\cdot,\cdot)$ on $\p$. The metric $g$ is called \emph{naturally reductive} if an $\Ad(H)$-invariant complement $\p$ can be chosen in such a way that $([X,Y]_{\p},X) = 0$ for all $X,Y \in \p$, where the subscript $\p$ denotes the $\p$-component. In this case, we say that the (naturally reductive) metric $g$ \emph{is generated by the pair} $(\p, (\cdot ,\cdot ))$. For comparison, on the Lie algebra level, $g$ is geodesic orbit if and only if for any $X \in \p$ (with any choice of $\p$), there exists $Z \in \h$ such that $([X+Z,Y]_{\p},X) =0$ for all $Y \in \p$ \cite[Proposition~2.1]{KV}. It immediately follows that any naturally reductive space is a $G$-geodesic orbit space; the converse is false when $\dim M \ge 6$. Clearly, the property of being naturally reductive depends on the choice of the group $G$ (the choice of the presentation $M=G/H$); both enlarging and reducing $G$ may result in gaining or losing this property. In this paper, the presentation $M=G/H$ (and hence the group $G$) will be fixed, and so ``naturally reductive" will always mean ``$G$-naturally reductive", unless explicitly stated otherwise.


\smallskip

Our setup in this paper is as follows. Let $M=G/H$ be a compact, connected, simply connected, Riemannian homogeneous space, with $G$ acting almost effectively (this means that any normal subgroup of $G$ contained in $H$ is discrete). We classify all the $G$-GO metrics on $M$, both naturally reductive and not, under the assumption that \emph{$H$ is a simple Lie group} (that is, any normal proper subgroup of $H$ is discrete). Note that $H$ is then closed (compact) and connected and $G$ is connected. Moreover, the fundamental group of $H$ must be finite, and since $G$ is compact (and hence reductive), with a finite fundamental group (from the exact sequence of the fibration $H \to G \to G/H$), $G$ must be a compact semisimple Lie group. 

\smallskip

We first characterise naturally reductive metrics on $G/H$. Let $\g=\oplus_{i=1}^N \g_i, \; N \ge 1$, be a decomposition of $\g$ into simple ideals. The inclusion $\h \hookrightarrow \g$ followed by the linear projection $\g \to \g_i$ relative to this decomposition defines a projection of $\h$ to each of $\g_i$, which is a homomorphism of Lie algebras. As $\h$ is simple, every such homomorphism is either trivial or injective. Relabel the ideals $\g_i$ in such a way that $\g = \bigoplus_{i=1}^{N_0} \g_i \oplus \bigoplus_{i=N_0+1}^{N_1} \g_i \oplus \bigoplus_{i=N_1+1}^{N} \g_i$, where $0 \le N_0 \le N_1 \le N, \; N_0 < N$, and the projection of $\h$ to $\g_i$ is trivial for $i=1, \dots, N_0$, is injective, but not surjective for $i=N_0+1, \dots, N_1$, and is bijective for $i=N_1+1, \dots, N$ (so that $\g_{N_1+1}, \dots, \g_N$ are isomorphic to $\h$). Denote by $\ip_i$ minus the Killing form on $\g_i$, for $i=1, \dots, N_0$. For $i=N_0+1, \dots, N$, denote by $\ip_i$ the (negative) multiple of the Killing form on $\g_i$ normalised in such a way that its restriction to the projection of $\h$ to $\g_i$ equals minus the Killing form on $\h$.

\begin{theorem} \label{th:natred}
  Let $M=G/H$ be a compact, connected, simply connected, Riemannian homogeneous space, where $G$ is almost effective and $H$ is a simple Lie group. An invariant metric on $M$ is \emph{(}$G$-\emph{)}naturally reductive if and only if it is generated by a pair $(\p, (\cdot ,\cdot ))$ such that, in the above notation,
    \begin{enumerate} [label=\emph{(\alph*)},ref=\alph*]
    \item \label{it:natredideal} 
        either $\p=\oplus_{i \ne j} \g_i$ is an ideal in $\g$, for some $j\in \{N_1+2,\cdots,N\}$ \emph{(}so that $\g_j$ is isomorphic to $\h$\emph{)}, and $(\cdot ,\cdot )$ is an $\ad(\p)$-invariant inner product on $\p$, that is, $(\cdot ,\cdot )=\sum_{i \ne j} \beta_i \ip_i$, where $\beta_i > 0$.

    \item \label{it:natredQ}
        or $\p$ is the orthogonal complement to $\h \subset \g$ relative to an $\ad(\g)$-invariant quadratic form $Q=\sum_{i=1}^{N} \gamma_i \ip_i$ on $\g$ and $(\cdot ,\cdot ) = Q_{|\p}$, where
        \begin{enumerate} [label=\emph{(\roman*)},ref=\roman*]
        \item \label{it:natredQ+}
        either $\gamma_i > 0$ for all $i=1, \dots, N$,

        \item \label{it:natredQ-}
        or there exists $j \in \{N_1+1, \dots, N\}$ such that $\gamma_j < 0$ and $\gamma_i > 0$ for all $i \ne j$, and $\sum_{i=N_0+1}^N \gamma_i < 0$. \end{enumerate}
    \end{enumerate}
\end{theorem}

\begin{remark} \label{rem:LO}
  Note that all the metrics from \eqref{it:natredideal} are reducible when $N >2$; this is not necessarily true for metrics in \eqref{it:natredQ}. Also note that if $\g$ contains no simple ideals isomorphic to $\h$, then any naturally reductive metric is normal (that is, it is a restriction to $\p$ of a bi-invariant metric on $\g$; a normal metric is always naturally reductive). Theorem~\ref{th:natred} generalises the result of \cite[Theorem~1]{NNLO} for \emph{Ledger-Obata spaces}. In fact, a Ledger-Obata space is the homogeneous space $G/H$ with $N_1=0$ in our notation.
\end{remark}

The classification of $G$-GO metrics which are not naturally reductive is given in the following Theorem.

\begin{theorem} \label{th:simple}
  Let $M=G/H$ be a compact, connected, simply connected, Riemannian homogeneous space, where $G$ is almost effective and $H$ is a simple Lie group. Suppose $M$ is a $G$-GO space. Then either $M$ is \emph{(}$G$-\emph{)}naturally reductive, or one of the following is true.
  \begin{enumerate} [label=\emph{(\Alph*)},ref=\Alph*]
    \item \label{it:thirred}
    If $M$ is an irreducible Riemannian manifold, then $G$ is simple and $M$ belongs to the following list, up to a finite cover \emph{(}the corresponding metrics are given in Table~\ref{t:gosimple}\emph{)}.
    \begin{enumerate}[label=\emph{(\arabic*)},ref=\arabic*]
        \item \label{it:97}
        $\SO(9)/\Spin(7)$; 
        \item \label{it:107}
        $\SO(10)/\Spin(7)$;
        \item \label{it:117}
        $\SO(11)/\Spin(7)$;
        \item \label{it:e610}
        $\E_6/\Spin(10)$;
        \item \label{it:susu}
        $\SU(n+p)/\SU(n), \; n \ge 2, \, 1 \le p \le n-1$;
        \item \label{it:so2su}
        $\SO(2n+1)/\SU(n), \; n \ge 3$;
        \item \label{it:so4su2}
        $\SO(4n+2)/\SU(2n+1), \; n \ge 2$;
        \item \label{it:spsp}
        $\Sp(n+1)/\Sp(n), \; n \ge 1$;
        \item \label{it:susp}
        $\SU(2n+1)/\Sp(n), \; n \ge 2$;
        \item \label{it:8g2}
        $\Spin(8)/\G_2$;
        \item \label{it:9g2}
        $\SO(9)/\G_2$.
  \end{enumerate}

    \item \label{it:thred}
    If $M$ is reducible, then it is the Riemannian product of one of the spaces in~\eqref{it:thirred} and a compact semisimple Lie group with a bi-invariant metric.
  \end{enumerate}
\end{theorem}

Note that the cases in \eqref{it:thirred} are mutually exclusive. We also note that many of these spaces already appeared in the literature. For example, the spaces \eqref{it:97}, \eqref{it:spsp} and \eqref{it:susu} with $p=1$ are spheres with a GO metric \cite{Nsp}; the spaces \eqref{it:97}, \eqref{it:e610}, \eqref{it:susp}, \eqref{it:8g2} and \eqref{it:susu} with $p=1$ are weakly symmetric \cite{Yak}; the spaces \eqref{it:97}, \eqref{it:e610}, \eqref{it:so4su2}, \eqref{it:8g2} and \eqref{it:susu} with $p=1$ are GO spaces with exactly two irreducible isotropy components \cite{CN2019}
\footnote{The paper \cite{CN2019} is based on the classification in \cite{DK}, which omits five cases given in \cite{He} and the case $\E_8/\Spin(9)$ given in \cite[Remark~6.1]{Lau}. But it is easy to check that the results of \cite{CN2019} still hold after taking these cases into account.}. Moreover, the space \eqref{it:so2su} with $n$ even (and several others from our list) is fibered over a compact symmetric space, with the GO metric having the property that its restriction to the tangent space of the fiber is proportional to the restriction of the Killing form on $G$ (so that the tangent space of the fiber at $eH$ is an eigenspace of the metric endomorphism --- see Section~\ref{ss:decomposition}) \cite{T1}. It should be noted that the GO metrics on the spaces \eqref{it:so2su} with $n$ even and with $n$ odd are very different --- see the details in Table~\ref{t:gosimple}; in particular, in the odd case, a nontrivial algebraic condition has to be satisfied.

We note that in the other ``extremal" case, when the isotropy subgroup is abelian, any GO metric is naturally reductive by the result of \cite{S20} (these two classifications very nicely complement one another; note that some authors include $\SO(2)$ in the list of simple groups).

\smallskip

The paper is organised as follows. In Section~\ref{s:pre} we provide the necessary background material and also give a detailed description of the GO metrics on the spaces listed in Theorem~\ref{th:simple}\eqref{it:thirred} (see Table~\ref{t:gosimple}). In Section~\ref{s:natred} we prove Theorem~\ref{th:natred}. In the rest of the paper we give the proof of Theorem~\ref{th:simple}. The proof is based on the study of different types of submodules in the decomposition of $\p$: trivial, \emph{large} and adjoint modules are considered in Section~\ref{s:tla}, and \emph{tiny} modules, in Section~\ref{s:tiny} (we refer to Section~\ref{s:pre} for unexplained terminology).

\section{Preliminaries}
\label{s:pre}

\subsection{Generalities}
\label{ss:decomposition}

Throughout the paper, we will adopt the assumptions of Theorem~1 and Theorem~2 (although some notions and facts below do not require all of them). Namely, we work with a compact, connected, simply connected, Riemannian homogeneous space $M=G/H$, where the Lie group $G$ acts almost effectively, and $H$ is a simple Lie group. As we noted above, $H$ is then compact and connected and $G$ is a compact, connected, semisimple Lie group.

Let $\g$ be the Lie algebra of $G$ and $\h \subset \g$ the Lie algebra of $H$. Denote by $\ip$ minus the Killing form on $\g$. Throughout the paper, ``orthogonal" means ``orthogonal relative to $\ip$" unless otherwise is explicitly stated. Let $\g=\h \oplus \p$ be an $\Ad(H)$-reductive decomposition (one possibility is to take $\p$ as the orthogonal complement $\m$ to $\h$ in $\g$). Then $\p$ can be naturally identified with the tangent space $T_{eH}(G/H)$, and the Riemannian metric $g$ is determined by some positive $\ip$-symmetric $\Ad(H)$-equivariant endomorphism $A:\p \to \p$ by the formula $g_{eH}(X,Y)=\<A X, Y\>$, for $X, Y \in \p$. We call $A$ the \emph{metric endomorphism}.

We have the following fact \cite[Proposition~1]{AA}, \cite[Proposition~2]{S18}. 
\begin{lemma}\label{GO-criterion}
A homogeneous Riemannian manifold $(M=G/H,g)$ with a semisimple group $G$ and an $\Ad(H)$-reductive decomposition $\g=\h \oplus \p$ is a $G$-geodesic orbit space if and only if for any $X \in \p$ there exists $Z \in \h$ such that
\begin{equation}\label{eq:GO}
  [X+Z,AX] = 0.
\end{equation}
\end{lemma}

Note that the claim of the Lemma does not depend on a particular choice of $\p$; in particular, one can take $\p$ to be the orthogonal complement $\m$ to $\h$. In the assumptions of the Lemma, we call any map $Z: \p \to \h$ such that $[X+Z(X),AX] = 0$ for all $X \in \p$, a \emph{geodesic graph}. In general, a geodesic graph may not be unique, but if it at all exists (that is, if $M$ is a $G$-GO space), then it can be chosen $\Ad(H)$-equivariant. One obvious, but potentially confusing point here is that although for almost all objects (decompositions, modules, inner products, etc.) throughout the paper $\Ad(H)$-invariancy/equivariancy is synonymous with $\ad(\h)$-invariancy/equivariancy respectively, this is not necessarily true for a geodesic graph: an $\Ad(H)$-equivariant geodesic graph does not have to be $\ad(\h)$-equivariant (the obvious reason being non-linearity --- cf. Lemma~\ref{l:linearGG} in Section~\ref{ss:nr}).

In the proof of Theorem~\ref{th:simple} below we choose and fix $\p$ to be the orthogonal complement $\m$ to $\h$ in $\g$. Let $\a_1, \dots, \a_m > 0$ be the (distinct) eigenvalues of the metric endomorphism $A$, and let $\m_1, \dots, \m_m$ be the corresponding eigenspaces. Each $\m_i$ is an $H$-module and the decomposition $\m=\m_1 \oplus \m_2 \oplus \dots \oplus \m_m$ is orthogonal and $\Ad(H)$-invariant. Since $H$ is connected, a submodule of $\m$ is $\Ad(H)$-irreducible (respectively $\Ad(H)$-invariant) if and only if it is $\ad(\h)$-irreducible (respectively $\ad(\h)$-invariant).

We can further decompose every submodule $\m_i$ in the decomposition $\m=\oplus_{i=1}^m \m_i$ into an orthogonal sum of irreducible modules. Labelling them through we get the orthogonal decomposition
\begin{equation}\label{eq:nirr}
\m=\oplus_{r=1}^p \n_r
\end{equation}
into irreducible $\h$-modules $\n_r$ each of which lies in some $\m_i$. Note that at least one of the modules $\n_r$ is nontrivial, as $G$ acts almost effectively (here and below, by a trivial module we mean a module on which the group/algebra acts trivially).

We call an $H$-module $\n \subset \m$ \emph{large} if the principal stationary subgroup of the action of $\Ad(H)$ on $\n$ is discrete. On the level of Lie algebras, $\n$ is large if for some $X \in \n$ (and then for all $X$ in an open and dense subset of $\n$), the centraliser $\z_\h(X)$ in $\h$ is trivial. In the context of GO spaces, if $\m$ is large (in particular, if one of its submodules is large), then by \eqref{eq:GO} the geodesic graph $Z$ is uniquely determined on an open, dense subset.

An $H$-module is called \emph{small} if it is not large. Clearly a trivial module is always small; the adjoint module is also small. Irreducible small modules for compact simple and semisimple Lie groups are given in \cite[Table~1~and~2]{HH}. In the case of simple groups, the list includes the adjoint representations, the standard representations of the classical groups, the ``defining" representations of the exceptional groups, two infinite series and five low-dimensional modules of the classical groups. Note that a small module can be the sum of more than one nontrivial irreducible submodules (see also \cite[Section~3]{Goz}). We call a module \emph{tiny} if it is small, irreducible, nontrivial and not adjoint; tiny modules are listed in Table~\ref{t:tiny}.

The GO condition imposes strong restrictions on the decompositions of $\m$ into the eigenspaces of $A$ and on the decomposition \eqref{eq:nirr}. 

\begin{lemma}[{\cite[Section~5]{Nik2017}}] \label{l:brackets}
In the assumptions of Theorem~\ref{th:simple} \emph{(}and in the above notation\emph{)}, we have the following.
\begin{enumerate} [label=\emph{(\alph*)},ref=\alph*]
  \item \label{it:alphabeta}
  For any $X \in \m_i, \, Y \in \m_j$, with $i \ne j$, there exists $Z \in \h$ such that $[X,Y]=\frac{\a_i}{\a_i-\a_j} [Z,X] + \frac{\a_j}{\a_i-\a_j} [Z,Y]$. So for $i \ne j$ we have $[\m_i, \m_j] \subset \m_i \oplus \m_j$, and, in the notation of \eqref{eq:nirr}, if $\n_r \subset \m_i, \, \n_s \subset \m_j$, then $[\n_r, \n_s] \subset \n_r \oplus \n_s$.

  \item \label{it:centralisers}
  For any $X \in \m_i, \, Y \in \m_j$, with $i \ne j$, there exist $Z_1 \in \z_\h(X), \, Z_2 \in \z_\h(Y)$ such that $[X,Y]= [Z_2,X] + [Z_1,Y]$.

  \item \label{it:large}
  Consequently, any two large modules $\m_i, \m_j$ with $i \ne j$ commute. Furthermore if $\m_i$ is large, then for any $\n_r \subset \m_i$ in the decomposition \eqref{eq:nirr} we have $[\m_i^\perp, \n_r] \subset \n_r$.

  \item \label{it:twoinone}
  If $X, Y  \in \m_i$ satisfy $[\h,X] \perp Y$, then $[X, Y] \in \m_i$.
\end{enumerate}
\end{lemma}

\begin{remark} \label{rem:decomp}
Note that the last inclusion in Lemma~\ref{l:brackets}\eqref{it:alphabeta} is a very powerful fact which will be used in many places in the proofs below (see also Remark~\ref{rem:n1n2triv}). Moreover, it defines an $\Ad(H)$-equivariant homomorphism $\n_r\times\n_s \rightarrow \n_r \otimes \n_s$, for all $r \ne s$. So in particular, if the irreducible decomposition of the $H$-module $\n_r\times\n_s$ contains no modules isomorphic to either $\n_r$ or $\n_s$, we get $[\n_r,\n_s]=0$. It would substantially simplify our arguments if a complete or even a partial classification of such pairs of modules would be known (for our purposes, we can assume that at least one of the modules is small), but we were not able to find it in the literature; in many cases in Section~\ref{s:tiny}, we use this condition for individual pairs of modules.
\end{remark}

\subsection{The table}
\label{ss:table}

In the above notation, we give an explicit description of the GO metrics in Theorem~\ref{th:simple}\eqref{it:thirred} in the table below (for the proofs, see Lemma~\ref{l:trivial}\eqref{it:onenont} and Section~\ref{s:tiny}).

{
\setlength{\tabcolsep}{3pt}
\renewcommand{\arraystretch}{1.5}
\begin{table}
  \centering
\begin{tabular}{|c|l|c|c|}
  \hline
      & $H \subset G$ & Modules $\m_i$ & Condition \\
  \hline
  (1) & $\Spin(7) \subset \SO(8) \subset \SO(9)$ & \parbox[t][][t]{6.8cm}{\raggedright $\m_1 \oplus \so(8) = \so(9), \, \m_2 \oplus \so(7) = \so(8)$ \\ $\m_1$ is spin, $\m_2=\br^7$} & $\a_1 \ne \a_2$ \\
  \hline
  (2) & $\Spin(7) \subset \SO(8) \subset \SO(10)$ & \parbox[t][][t]{6.8cm}{\raggedright $\m_1 \oplus \so(8) = \so(10), \m_2 \oplus \so(7) =\! \so(8)$\\ $\m_1= 2(\text{spin}) \oplus \br$, $\m_2 = \br^7$} & $\a_1 \ne \a_2$ \\
  \hline
  (3) & $\Spin(7) \subset \SO(8) \subset \SO(11)$ & \parbox[t][][t]{6.8cm}{\raggedright $\m_1 \oplus \so(8) = \so(11), \m_2 \oplus \so(7) =\! \so(8)$\\ $\m_1= 3(\text{spin}) \oplus \so(3)$, $\m_2=\br^7$} & $\a_1 \ne \a_2$ \\
  \hline
  (4) & $\Spin(10) \subset \Spin(10)\SO(2) \subset \E_6$ & \parbox[t][][t]{6.8cm}{\raggedright $\m_1 \oplus \so(10) \oplus \br = \eg_6$, \\ $\m_2 = \br$} & $\a_1 \ne \a_2$ \\
  \hline
  (5) & \pbox[t][][t]{10cm}{$\SU(n) \subset \mathrm{S}(\mathrm{U}(n) \times \mathrm{U}(p)) \subset \SU(n+p)$ \\ $n \ge 2, \; 1 \le p \le n-1$}& \parbox[t][][t]{6.8cm}{\raggedright $\m_1 = \su(p) \oplus p \bc^n$, \\ $\m_2 = \br$} & $\a_1 \ne \a_2$ \\
  \hline
  ($6_1$) & \pbox[t][][t]{10cm}{$\SU(n) \subset \mathrm{U}(n) \subset \SO(2n) \subset \SO(2n+1)$ \\ $n \ge 3, \, n$ is odd}& \parbox[t][][t]{6.8cm}{\raggedright $\m_2 \oplus \su(n) \oplus \br = \so(2n)$, \\ $\m_1 = \bc^n, \; \m_3 = \br$} & \pbox[t][][t]{3cm}{\centering $n\a_3^{-1} = (n-1)\a_2^{-1}$ \linebreak $+ \a_1^{-1}, \; \a_1 \ne \a_2$} \\
  \hline
  ($6_2$) & \pbox[t][][t]{10cm}{$\SU(n) \subset \mathrm{U}(n) \subset \SO(2n) \subset \SO(2n+1)$ \\ $n \ge 4, \, n$ is even}& \parbox[t][][t]{6.8cm}{\raggedright $\m_2 \oplus \su(n) = \so(2n)$, \\ $\m_1 = \bc^n$} & $\a_1 \ne \a_2$ \\
  \hline
  (7) & \pbox[t][][t]{10cm}{$\SU(2n+1) \subset \mathrm{U}(2n+1) \subset \SO(4n+2)$ \\ $n \ge 2$}& \parbox[t][][t]{6.8cm}{\raggedright $\m_1 \oplus \su(2n+1) \oplus \br = \so(4n+2)$, \\ $\m_2 = \br$} & $\a_1 \ne \a_2$ \\
  \hline
  (8) & \pbox[t][][t]{10cm}{$\Sp(n) \subset \Sp(n) \times \Sp(1) \subset \Sp(n+1)$ \\ $n \ge 1$}& \parbox[t][][t]{6.8cm}{\raggedright $\m_1 = \bH^n$, \\ $\m_2 = \sp(1), \; \m_2$ is trivial} & $\a_1 \ne \a_2$ \\
  \hline
  (9) & \pbox[t][][t]{10cm}{$\Sp(n) \subset \SU(2n) \subset \SU(2n+1)$ \\ $n \ge 2$}& \parbox[t][][t]{6.8cm}{$\m_2 \oplus \sp(n) = \su(2n)$, \\ $\m_1 = \bH^n = \bc^{2n}, \; \m_3 = \br$} & \pbox[t][][t]{3cm}{\centering not \linebreak ($\a_1 = \a_2 =\a_3$)} \\
  \hline
  (10) & \pbox[t][][t]{10cm}{$\G_2 \subset \Spin(7) \subset \Spin(8)$}& \parbox[t][][t]{6.8cm}{$\m_1, \m_2$ are any two isomorphic $7$-di\-mensional $\g_2$-modules in $\g_2^\perp \subset \so(8)$} & $\a_1 \ne \a_2$ \\ 
  \hline
  (11) & \pbox[t][][t]{10cm}{$\G_2 \subset \SO(7) \subset \SO(9)$}& \parbox[t][][t]{6.8cm}{$\m_2 \oplus \g_2 = \so(7)$, \\ $\m_1 \oplus \so(7) = \so(9), \; \m_1 = 2\br^7 \oplus \br$} & $\a_1 \ne \a_2$ \\
  \hline
\end{tabular}
\vskip 0.4cm
  \caption{GO metrics on the spaces in Theorem~\ref{th:simple}\eqref{it:thirred}}\label{t:gosimple}
\end{table}
}

In the table, the direct sum always means the orthogonal direct sum; note that in many cases, the modules $\m_i$ (the eigenspaces of the metric endomorphism $A$) are reducible. The condition on the eigenvalues of $A$ in the last column in all but one case simply says that for any positive $\a_1, \dots, \a_m$, the resulting metric is GO, and that it is naturally reductive only when $A$ is a multiple of the identity (so that the metric is normal); the only exception is case~(\ref{it:so2su}${}_1$): the eigenvalues of the metric endomorphism $A$ of a GO metric on the space $\SO(2n+1)/\SU(n)$ where $n \ge 3$ is odd, have to satisfy a certain algebraic condition.

The dimension of the space of GO metrics which are not naturally reductive in all the cases except two is $2$ (but note that in case~(\ref{it:so2su}${}_1$), $A$ has three eigenspaces). The exceptions are case~\eqref{it:susp} where the dimension is $3$ and $A$ can have two or three eigenspaces, and case~\eqref{it:8g2} where $A$ has two eigenspaces which are isomorphic as $\h$-modules, and so the dimension of the space of GO metrics is again $3$ (any invariant metric is GO \cite{Zil96}). 


\subsection{Natural reductivity}
\label{ss:nr}

It is easy to see that in the assumptions of Lemma~\ref{GO-criterion}, the space $(M=G/H,g)$ is \emph{naturally reductive} if there exists an $\Ad(H)$-reductive decomposition $\g=\h \oplus \p$ such that $[X,AX]=0$. It follows that any $G$-GO space is naturally reductive; the converse is true when $\dim M \le 4$, but is false starting from dimension $5$. More precisely, in \cite{KV} the authors constructed examples of $G$-GO spaces of dimension $5$ which are not naturally reductive but can made be such by choosing a different transitive isometry group acting on $M$; further, in dimension $6$ there are examples of GO spaces which are not naturally reductive, for any choice of a transitive isometry group.

The following fact, which is a stronger version of \cite[Proposition~2.10]{KV} (see also \cite{Sz}), will be useful to detect whether a GO space is naturally reductive, without the necessity to produce a specific $\p$.

\begin{lemma} \label{l:linearGG} 
Suppose a homogeneous Riemannian manifold $(M=G/H,g)$ is a $G$-geodesic orbit space. Then $(M,g)$ is naturally reductive if and only if for some \emph{(}and then for any\emph{)} $\Ad(H)$-reductive decomposition $\g=\h \oplus \p$, there is a geodesic graph $Z: \p \to \h$ which is linear.
\end{lemma}
\begin{proof}
  The claim will follow from \cite[Proposition~2.10]{KV} if we can show that the existence of a linear geodesic graph implies the existence of a linear $\Ad(H)$-equivariant geodesic graph. This is similar to the fact that a geodesic graph can be always chosen $\Ad(H)$-equivariant (see the paragraph after Lemma~\ref{GO-criterion} in Section~\ref{s:intro}).

  Suppose we have a linear geodesic graph $Z:\p \to \h$. It is easy to see that for any $h \in H$, the map $X \mapsto \Ad_hZ(\Ad_{h^{-1}}X)$ is also a geodesic graph. But a convex linear combination of geodesic graphs is also a geodesic graph. Integrating the latter expression over $H$ by the Haar measure $\mu$ such that $\mu(H) = 1$ we obtain an $\Ad(H)$-equivariant geodesic graph $X \mapsto Z'(X):=\int_H \Ad_hZ(\Ad_{h^{-1}}X) d\mu(h)$. Note that $Z'$ is linear as $Z$ is such. 
\end{proof}



\section{Naturally reductive spaces. Proof of Theorem~\ref{th:natred}}
\label{s:natred}

In this section we prove Theorem~\ref{th:natred} following the approach in \cite[Section~3]{NNLO} (which corresponds, in our notation, to the special case $N_1=0$).


Suppose $\p$ is an $\ad(\h)$-invariant complement to $\h$ in $\g$. Then the space $\q=[\p,\p]+\p$ is an ideal in $\g$. 

By \cite[Theorem~4]{K56}, if a naturally reductive metric is generated by a pair $(\p, (\cdot ,\cdot ))$, then there is a (unique) $\ad(\q)$-invariant, non-degenerate quadratic form $Q$ on $\q$ such that
\begin{equation}\label{eq:Kostant}
    Q(\p, \q \cap \h)=0 \qquad \text{and} \qquad Q_{|\p}=(\cdot ,\cdot ).
\end{equation}
The converse is also true: if $Q$ is an $\ad(\q)$-invariant, non-degenerate quadratic form which satisfies the first equation of \eqref{eq:Kostant} and whose restriction to $\p$ is positive definite, then that restriction defines a naturally reductive metric; this follows from $\ad(\q)$-invariancy of $Q$ and from the fact that $\q$ is complemented in $\g$ by an ideal. 

We clearly have $\p + \h = \g$, and so in the notation of Section~\ref{s:intro}, there are only two possible cases.

In the first case, we have $\q = \oplus_{i=1}^{N-1} \g_i$ (up to relabelling the modules $\g_{N_1+1}, \dots, \g_N$), and then the (linear) projection of $\h \subset \g$ to $\g_N$ is an isomorphism (so in particular, $N_1 < N$). Then $\q \cap \h = 0$ and $\g=\q \oplus \h$, and so $\p = \q$, an ideal. Furthermore, $Q= (\cdot ,\cdot )$ by the second equation of \eqref{eq:Kostant} and has the form given in \eqref{it:natredideal}.

In the second case, $\q=\g$. The quadratic form $Q$ is $\ad(\g)$-invariant, and so we have $Q=\sum_{i=1}^{N} \gamma_i \ip_i$, with $\gamma_i \ne 0$. Then by \eqref{eq:Kostant} the space $\p$ is the $Q$-orthogonal complement to $\h$ in $\g$. Note that $N-N_0 \ge 2$, for if $N=N_0+1$, we get $\h=\g_N$ and then $\p=\q=\oplus_{i=1}^{N-1} \g_i$. We need to choose $\gamma_i$ in such a way that the restriction of $Q$ to $\p$ is positive definite. Clearly, for $i=1, \dots, N_0$ we have $\g_i \subset \p$, and so $\gamma_i > 0$. Similarly, we must have $\gamma_i > 0$ for $i=N_0+1, \dots, N_1$, as every such $\g_i$ contains a nonzero vector $Q$-orthogonal to $\h$. The restriction of $Q$ to $\p$ will obviously be positive definite if $\gamma_i > 0$ for all $i=N_1+1, \dots, N$; then $Q$ itself is positive definite and so the metric $(\cdot ,\cdot )$ is normal; this gives case \eqref{it:natredQ}\eqref{it:natredQ+}. Suppose $Q$ is indefinite. If $\gamma_j, \gamma_k < 0$ for some $j, k \in \{N_1+1, \dots, N\}, \; j \ne k$, then $Q$ is negative definite on the $Q$-orthogonal complement to $\h$ in $\g_j \oplus \g_k$. It therefore remains to consider the case when $\gamma_j < 0$ for exactly one $j \in \{N_1+1, \dots, N\}$. Up to relabelling we can assume that $j=N$ and so $\gamma_i > 0$ for all $i < N$. Identify the images of the (linear) projections of $\h$ to $\g_i, \; i=N_0+1, \dots, N$, with $\h$ (recall that each of these projections is an isomorphism on its image and that every inner product $\ip_i$ is normalised in such a way that the corresponding projection is a linear isometry). Then the restriction of $Q$ to $\p$ is positive definite if and only if for any $Y_{N_0+1}, \dots, Y_N \in \h$ such that $\sum_{i=N_0+1}^{N} \gamma_i Y_i = 0$ and not all $Y_i$ are zeros, we have $\sum_{i=N_0+1}^{N} \gamma_i \|Y_i\|^2 > 0$ (where the norm is computed relative to minus the Killing form on $\h$). Equivalently, $\sum_{i=N_0+1}^{N-1} \gamma_i \|Y_i\|^2 + \gamma_N^{-1}\|\sum_{i=N_0+1}^{N-1} \gamma_i Y_i\|^2 > 0$ when at least one of $Y_i$ is nonzero. Taking $Y_{N_0+1}= \dots = Y_{N-1} \ne 0$ we obtain a necessary condition $\gamma_N + \sum_{i=N_0+1}^{N-1} \gamma_i < 0$. But this condition is also sufficient, as from $\sum_{k=N_0+1}^{N-1} \gamma_k < -\gamma_N$ we obtain $(\sum_{k=N_0+1}^{N-1} \gamma_k)(\sum_{i=N_0+1}^{N-1} \gamma_i \|Y_i\|^2 + \gamma_N^{-1}\|\sum_{i=N_0+1}^{N-1} \gamma_i Y_i\|^2) \ge  (\sum_{k=N_0+1}^{N-1} \gamma_k)(\sum_{i=N_0+1}^{N-1} \gamma_i \|Y_i\|^2) - \|\sum_{i=N_0+1}^{N-1} \gamma_i Y_i\|^2 \ge 0$, by the Cauchy-Schwarz inequality, with the equality only possible when $Y_{N_0+1}= \dots = Y_{N-1}$ and $\sum_{i=N_0+1}^{N-1} \gamma_i Y_i = 0$, that is, when all $Y_i$ are zeros. This gives the condition in \eqref{it:natredQ}\eqref{it:natredQ-} and completes the proof.

\begin{remark} \label{rem:nrm}
  Using Theorem~\ref{th:natred} one can easily write down the inner product and the metric endomorphism $A$ on the orthogonal complement $\m$ to $\h$ in $\g$. Note that first every $\gamma_i, \; i=N_0+1, \dots, N_1$, has to be re-scaled by the ratio of the restriction of the Killing form of $\g_i$ to the projection of $\h$ to $\g_i$ and the Killing form of $\h$.
\end{remark}

\section{$G$-GO spaces. Trivial, large and adjoint modules}
\label{s:tla}

In this section, we study trivial, large and adjoint modules in the decomposition of $\m$. In the next section, we will complete the proof of Theorem~\ref{th:simple} by a case-by-case study of tiny modules.

\emph{Throughout both sections \emph{(}in particular, in all the Lemmas and the Propositions\emph{)}, we adopt the assumptions of Theorem~\ref{th:simple} and we use the terminology and the notation introduced in Section~\ref{s:pre}.}

Recall that $M=G/H$ is a compact, connected, simply connected, Riemannian homogeneous space, with $G$ acting almost effectively and $H$ being a simple Lie group. Then $G$ is compact, connected and semi-simple and $H$ is compact and connected. We denote by $\ip$ minus the Killing form on $\g$ and consider the orthogonal decomposition $\g=\h \oplus \m$, where the $\h$-module $\m$ is identified with the tangent space of $M$ at $eH$. The metric on $M$ is defined by the metric endomorphism $A$ whose eigenspaces $\m_i, \; i=1, \dots, m$, are orthogonal $\h$-modules with corresponding eigenvalues $\a_i > 0$. We will use Lemmas~\ref{GO-criterion} and \ref{l:linearGG} to check when $M$ is a $G$-GO space and when $M$ is not naturally reductive respectively.

Our strategy is to consider the decomposition of $\m=\oplus_{i=1}^m \m_i$ into the eigenspaces $\m_i$ of $A$ together with the ``finer" decomposition \eqref{eq:nirr}. In this section, we will first study the trivial modules which may occur in \eqref{eq:nirr}. Next we show that if at least one module in the decomposition \eqref{eq:nirr} is either large or adjoint, then any $G$-GO metric is naturally reductive (with two exceptions, $\SU(3)/\SU(2)$ and $\Sp(2)/\Sp(1)$). Our main tools will be Lemma~\ref{GO-criterion}, Lemma~\ref{l:brackets} and Lemma~\ref{l:linearGG}. It will therefore follow that a $G$-GO space may be not naturally reductive only when all the modules in the decomposition \eqref{eq:nirr} are either trivial or tiny (recall that this means that the module is irreducible, nontrivial, not adjoint, and that its generic element has a nontrivial centraliser in $\h$). We will then use the classification of such modules given in \cite[Table~1]{HH} to complete the proof of Theorem~\ref{th:simple}.

We can assume that $m > 1$, as otherwise the metric is normal and hence naturally reductive. Another easy but useful observation is as follows.

\begin{lemma} \label{l:red}
  In the assumptions of Theorem~\ref{th:simple}, suppose that one of the eigenspaces of $A$ contains a nonzero ideal of $\g$. Then $M$ is the Riemannian product of a compact, simply connected Lie group with a bi-invariant metric and a compact, connected, simply connected homogeneous space $\hat{M}=\hat{G}/H$. Moreover, $M$ is a $G$-GO space \emph{(}respectively $G$-naturally reductive\emph{)} if and only if $\hat{M}$ is a $\hat{G}$-GO space \emph{(}respectively $\hat{G}$-naturally reductive\emph{)}.
\end{lemma}
\begin{proof}
  We can assume that in the presentation $M=G/H$ both $G$ and $H$ are simply connected (by replacing, if necessary, $G$ by its universal cover and $H$ by the identity component of its full preimage under the covering map). It is sufficient to prove the lemma when the ideal is simple. Let $\g=\oplus_{l=1}^N \g_l$ be the decomposition of $\g$ into simple ideals and suppose $\g_k \subset \m_i$ for some $i=1, \dots, m$. Denote $\hat{\g}=\g_k^\perp$ and $\hat{\m}=\g_k^\perp \cap \m$. We compute the linear holonomy algebra of $M$ using the construction in \cite{K}. Extend the metric endomorphism $A$ to the operator $C$ on $\g$ which is symmetric relative to $\ip$ and is defined by $C_{|\m}=A$ and $C_{|\h}=0$. For $Z \in \g$ define $D_Z:\m \to \m$ by $D_Z (Y)= [Z, Y]_{\m}$, for $Y \in \m$, where the subscript $\m$ denotes the $\m$-component. Then by {\cite[Theorem~2.3]{K}}, the linear holonomy algebra of $M=G/H$ is the Lie algebra generated by all the operators on $\m$ of the form
    \begin{equation*}
    \Gamma_Z=D_Z+C^{-1}D_Z C- C^{-1}D_{CZ}, \qquad Z \in \g.
    \end{equation*}
  (note that $C^{-1}$ is only applied to elements of $\m$). As $\g_k \subset \g$ is an ideal and is $C$-invariant we obtain that $\Gamma_Z (\g_k) \subset \g_k$, for all $Z \in \g$, and so the linear holonomy algebra preserves the orthogonal decomposition $\m = \hat{\m} \oplus \g_k$. As $M$ is simply connected, it is the Riemannian product $\hat{M} \times M_k$, where the tangent spaces to $\hat{M}$ and $M_k$ at $eH$ are $\hat{\m}$ and $\g_k$ respectively. Now $\g_k$ is an ideal orthogonal to (and hence commuting with) $\h$ and the restriction of $A$ to $\g_k$ is a multiple of the identity. It follows that $M_k$ is the simply connected, compact Lie group with the Lie algebra $\g_k$ and with a bi-invariant metric. Furthermore, $\h \subset \hat{\g}$ and so $H \subset \hat{G}$, where $\hat{G}$ is the compact, simply connected Lie group with the Lie algebra $\hat{\g}$. Then $\hat{M}=\hat{G}/H$, with the metric defined (relative to minus the Killing form of $\hat{\g}$) by the metric endomorphism $\hat{A}$ which is the restriction of $A$ to $\hat{\m}$.

  Now if $M$ is a $G$-GO space, then $\hat{M}$ is $\hat{G}$-GO (and the converse is also true). One way to see that is to define the geodesic graph $\hat{Z}: \hat{\m} \to \h$ (see Lemma~\ref{GO-criterion}) by restricting a geodesic graph $Z: \m \to \h$ to $\hat{\m}$ and using the fact that $[\h, \g_k] = [\hat{\m}, \g_k] = 0$. By Lemma~\ref{l:linearGG} this also implies that if $M$ is naturally reductive, then $\hat{M}$ also is (relative to $\hat{G}$).
\end{proof}

\subsection{Trivial modules}
\label{ss:trivial}

For every eigenvalue $\a_i$ of $A$, denote by $\t_i$ the maximal $\h$-trivial submodule of $\m_i$, and $\m'_i$ its orthogonal complement in $\m_i$. Let $\t=\oplus_{i=1}^m \t_i$ and $\m'=\oplus_{i=1}^m \m'_i$. We will need the following Lemma (note that statement \eqref{it:realS} is well known \cite{Mal}).

\begin{lemma} \label{l:trivial}
In the assumptions of Theorem~\ref{th:simple}, suppose the metric is GO. In the above notation, we have the following.
\begin{enumerate} [label=\emph{(\alph*)},ref=\alph*]
  \item \label{it:Vaideals}
  The submodules $\t_i$ are commuting, reductive ideals of the subalgebra $\t = \z_\g(\h) \subset \g$. 

  \item \label{it:actionofVa}
  For any $i=1, \dots, m$, we have $[\t, \m'_i] \subset \m'_i$. The restrictions of $\ad_\t$ and $\ad_\h$ commute on every $\m'_i$. Moreover, $\ad_{\t_i}$  respects any irreducible decomposition of $\m'_j$ if $j \ne i$, and the decomposition of $\m'_i$ into the sums of isomorphic modules.

  \item \label{it:nonz}
  If an irreducible submodule $\n \subset \m'$ is $\ad_T$-invariant for $T \in \t$, then either $[T,\n]=0$ or $(\ad_T^2)_{|\n} = \lambda \, \id_{|\n}$ for some $\lambda < 0$. In the latter case, $[\n,\n] \perp \n$.

  \item \label{it:realS} 
  The Lie algebra of skew-symmetric operators on an irreducible module $\n$ which commute with the restriction of $\ad_\h$ to $\n$ is either trivial, or is isomorphic to one of $\so(2)$ or $\so(3)$; then the module $\n$ is said to be of \emph{real}, \emph{complex} or \emph{quaternionic type}, respectively. In particular, the adjoint module is of real type.
\end{enumerate}
Suppose additionally that no eigenspace $\m_i$ contains a nonzero ideal of $\g$ \emph{(}cf. Lemma~\ref{l:red}\emph{)}. Then we have the following.
\begin{enumerate} [resume*]
  \item \label{it:noideal}
  $[T,\m'] \ne 0$, for any nonzero $T \in \t$.

  \item \label{it:simple}
  If no irreducible module $\n_r$ in the decomposition \eqref{eq:nirr} is adjoint, then $\g$ is simple.

  \item \label{it:onenont} 
  If no more than one irreducible module $\n_r$ in the decomposition \eqref{eq:nirr} is nontrivial, then either the metric is normal, or $M$ is one of the following spaces:
  \begin{gather*}
  \SO(4n + 2)/\SU(2n + 1), \; n \ge 2, \quad \SU(n + 1)/\SU(n), \; n \ge 2, \\
  \E_6/\Spin(10), \quad \Sp(n + 1)/\Sp(n), \; n \ge 1.
  \end{gather*}
\end{enumerate}
\end{lemma}
The spaces in \eqref{it:onenont} are spaces of cases \eqref{it:so4su2}, \eqref{it:susu} with $p=1$, \eqref{it:e610} and \eqref{it:spsp} of Theorem~\ref{th:simple}\eqref{it:thirred} respectively (cf. \cite[Theorem~2]{CN2019}).

\begin{remark} \label{rem:n1n2triv}
  As a side remark we note that assertion \eqref{it:actionofVa} combined with Lemma~\ref{l:brackets}\eqref{it:alphabeta} imposes further restrictions on the brackets of nontrivial irreducible submodules of $\m$ (similar to the second statement in \eqref{it:nonz}). For example, if $\n_1$ and $\n_2$ are irreducible submodules lying in different eigenspaces $\m_i$ and $\m_j$ of $A$ and $\ad_\t$ acts nontrivially on $\n_1 \oplus \n_2$, then either $[\n_1, \n_2] \subset \n_1$ or $[\n_1, \n_2] \subset \n_2$. 
\end{remark}

\begin{proof}
\eqref{it:Vaideals} Clearly $\t$ is the centraliser of $\h$ and hence is a subalgebra in $\g$. Furthermore, by Lemma~\ref{l:brackets}\eqref{it:centralisers} we have $[\t_i, \t_j]=0$ for $i \ne j$ and by Lemma~\ref{l:brackets}\eqref{it:twoinone}, $[\t_i, \t_i] \subset \m_i \cap \t=\t_i$.

\eqref{it:actionofVa} By Lemma~\ref{l:brackets}\eqref{it:alphabeta} we have $[\t_i, \m'_j] \subset \m'_j$ for $i \ne j$, and moreover, irreducible $\h$-submodules of $\m'_j$ are $\ad(\t_i)$-invariant. Furthermore, $[\m'_i, \t_i] \subset \m'_i$ by Lemma~\ref{l:brackets}\eqref{it:twoinone} and from the fact that $\t_i$ is a subalgebra. The restrictions of $\ad_\t$ and $\ad_\h$ clearly commute on every $\m'_i$; then by Schur's Lemma, $\ad_{\t_i}$ preserves the isotopic components of $\m'_i$.

\eqref{it:nonz} The first statement is obvious, as $\n$ is irreducible and $(\ad_T^2)_{|\n}$ is a symmetric operator commuting with $\ad_\h$ by \eqref{it:actionofVa}. To prove the second statement, consider the three-form $\omega \in \Lambda^3 \n$ defined by $\omega(X,Y,Z)=\<[X,Y],Z\>$ for $X, Y, Z \in \n$. Then $\Ad(\exp(\pi (-\la)^{-1/2}T))$ acts on $\n$ as $-\id_{|\n}$ and leaves $\omega$ invariant, so $\omega=0$.

\eqref{it:realS} Both statements are well known. 

\eqref{it:noideal} Suppose the centraliser $\k$ of $\m'$ in $\t$ is nontrivial. Then by \eqref{it:actionofVa}, $\k$ is an ideal in $\t$ and hence also in $\g$, as $[\k, \h \oplus \m']=0$. Then $\k \cap \t_i, \; i=1, \dots, m$, is an ideal of $\g$ by \eqref{it:Vaideals}. As $\k \cap \t_i \subset \m_i$, it must be zero by our assumption. It follows that $[\k, \t_i]=0$, for all $i=1, \dots, m$, and so $\k$ lies in the centre of $\g$ contradicting the fact that $\g$ is semisimple.

\eqref{it:simple} Let $\g=\oplus_{l=1}^N \g_l$ be the decomposition of $\g$ into simple ideals. As $\h$ is simple, the projection of $\h$ to each of them is either a monomorphism or trivial, and it is nontrivial for at least one $l$, say for $l=N$. If the projection to any other ideal $\g_l, \, l < N$, is also nontrivial, then $\m$ contains an adjoint module, and hence so does \eqref{eq:nirr}. Otherwise, $\h \subset \g_N$ and so $\g'=\oplus_{l=1}^{N-1} \g_l$ lies in $\m$ and is a trivial $\h$-module. Then $\g' \subset \t$ by \eqref{it:Vaideals} and so $\m' \subset \g_N$. It follows that $[\g',\m']=0$, and therefore $\g'=0$ by \eqref{it:noideal}.

\eqref{it:onenont} If all the modules in \eqref{eq:nirr} are trivial (that is, if $\m$ is a trivial module), then say $\t_1=\m_1$ is an ideal of $\g$ by \eqref{it:Vaideals}, a contradiction. Let $\n=\n_1 \subset \m_1$ be the only nontrivial irreducible module. Then $\m = \n \oplus \t$, and so by \eqref{it:actionofVa}, \eqref{it:noideal} and \eqref{it:realS}, we have $\dim \t \in \{0,1,3\}$.

If $\t=0$, then $\m=\n=\m_1$ and so the metric is normal. If $\dim \t=1$, the module $\t$ cannot lie in $\m_1$ (otherwise the metric is normal), so $\m$ has exactly two irreducible submodules, $\m_1=\n$ and $\m_2=\t_2=\br$. By \eqref{it:noideal} and \eqref{it:realS}, the module $\n$ is not adjoint, and so $\g$ is simple by \eqref{it:simple}. Consulting the classification in \cite[Theorem~2]{CN2019} we find that the spaces $M$ for which $\h$ is simple and one of the irreducible submodules in $\m$ is one-dimensional are $\SO(4n + 2)/\SU(2n + 1)$ for $n \ge 2$, $\SU(n + 1)/\SU(n)$ for $n \ge 2$ and $\E_6/\Spin(10)$.

Now suppose $\dim \t = 3$. Then by \eqref{it:actionofVa}, \eqref{it:noideal} and \eqref{it:realS}, $\t$ is a subalgebra isomorphic to $\so(3)$ and so by \eqref{it:Vaideals} it is a single ideal $\t_i$. That ideal cannot lie in $\m_1$ (as otherwise the metric is normal) and so we have $\m=\m_1 \oplus \m_2$, where $\m_1=\n$ and $\m_2=\t_2$. Similar to the above, $\g$ is simple and $\t_2$ acts on $\n$ nontrivially. But then $\h'=\h \oplus \t_2$ is a subalgebra of $\g$ and $(\g, \h')$ is a symmetric pair by \eqref{it:nonz}. Examining the list in \cite[Theorem~4.1]{T1} we find that $M=\Sp(n + 1)/\Sp(n), \; n \ge 1$.
\end{proof}

\subsection{Large modules}
\label{ss:large}

Recall that a module (not necessarily irreducible) is said to be large if it contains an element whose centraliser in $\h$ is trivial. It is clear that the set of such elements in a large module is open and dense and that a module is large if its submodule is large. Moreover, from \eqref{eq:GO} it follows that a geodesic graph is uniquely defined on an open, dense subset of a large module. As usual, we adopt the assumptions of Theorem~\ref{th:simple} and use the notation of Section~\ref{s:pre}.

We start with the following technical fact.

\begin{lemma} \label{l:non+large}
Suppose a module $\n:=\n_1 \subset \m_i$ in the decomposition \eqref{eq:nirr} is nontrivial and that its orthogonal complement $\n' = \oplus_{r=2}^p \n_r$ in $\m$ is a large module. Denote by $\mU \subset \n'$ the \emph{(}open and dense\emph{)} set of those elements whose centraliser in $\h$ is trivial. Then the restriction of a geodesic graph $Z$ to the subset $\mU \times \n$ is uniquely determined and there exist a unique linear map $L: \n' \to \h$ and a unique map $\Omega: \mU \to \Lin(\n, \h)$ such that for all $X \in \mU, \, Y \in \n$, we have
\begin{equation}\label{eq:Znon+l}
  Z=Z(X+Y)=LX+\Omega(X)Y.
\end{equation}
Moreover, for all $X \in \mU$ and $Y \in \n$, we have
\begin{gather}\label{eq:OmegaYY}
  [\Omega(X)Y,Y] = 0, \\
  [AX,X] = [LX, X], \label{eq:LXX} \\
  \a_i[LX, Y] + [\Omega(X)Y,AX]=[AX-\a_i X,Y]. \label{eq:non+l}
\end{gather}
\end{lemma}
\begin{proof}
Let $X \in \mU \subset \n', \, Y \in \n$. Applying equation~\eqref{eq:GO} to $X+Y$ we get
\begin{equation} \label{eq:non+lGO}
[Z(X+Y),AX] + \a_i[Z(X+Y),Y] + [X,AX] + [Y,AX-\a_i X] = 0.
\end{equation}
Note that the geodesic graph $X+Y \mapsto Z(X+Y)$ is uniquely defined on $\mU \times \n$. Let $F$ be the restriction of $Z$ to $\mU$. Taking $Y=0$ in \eqref{eq:non+lGO} we obtain
\begin{equation}\label{eq:FXAX}
[F(X),AX] + [X,AX] = 0,
\end{equation}
for $X \in \mU$, and so $[Z(X+Y)-F(X),AX] + \a_i [Z(X+Y),Y] + [Y,AX-\a_i X] = 0$ from \eqref{eq:non+lGO}. Projecting this equation to $\n'$ we obtain $[Z(X+Y)-F(X),AX] + \pi_{\n'}[Y,AX-\a_i X] = 0$, where $\pi_{\n'}:\m \to \n'$ is the orthogonal projection (the fact that $[Y,AX-\a_i X] \in \m$ follows from Lemma~\ref{l:brackets}\eqref{it:alphabeta}, as $AX-\a_i X \perp \m_i$). From the fact that the second term is linear in $Y$ and that $Z(X+Y)$ and $F(X)$ are uniquely determined we find that the element $Z(X+Y)-F(X) \in \h$ depends linearly on $Y$, for every $X \in \n'$. Therefore there exists a map $\Omega: \mU \to \Lin(\n, \h)$ such that for all $X \in \mU, \, Y \in \n$, we have $Z(X+Y)=F(X)+\Omega(X)Y$. Substituting into \eqref{eq:non+lGO} and using \eqref{eq:FXAX} we obtain $[\Omega(X)Y,AX] + \a_i [F(X)+\Omega(X)Y,Y] + [Y,AX-\a_i X] = 0$. Considering the left-hand side, for every fixed $X \in \mU$, as a polynomial in $Y$ we obtain \eqref{eq:OmegaYY}. To prove \eqref{eq:LXX} and \eqref{eq:non+l} it remains to show that the map $F: \mU \to \h$ is in fact linear. Projecting the latter equation to $\n$ we find $[F(X),Y] + \pi_\n [Y,\a_i^{-1} AX-X] = 0$, and so $[F(X_1+X_2)-F(X_1)-F(X_2),Y] = 0$, for all $Y \in \n$ and for all $X_1, X_2 \in \mU$ such that $X_1+X_2 \in \mU$. But the centraliser of $\n$ in $\h$ is an ideal which must be trivial, as $\h$ is simple and $\n$ is a nontrivial module. It follows that $F(X_1+X_2)=F(X_1)+F(X_2)$, for an open, dense set of pairs $(X_1, X_2) \in \n' \times \n'$. The fact that $F$ is homogeneous of degree $1$ in $X \in \mU$ follows from \eqref{eq:FXAX}. Therefore there exists a linear map $L: \n' \to \h$ whose restriction to $\mU$ coincides with $F$.
\end{proof}

Note that from \eqref{eq:non+l} it follows that the map $\Omega: \mU \to \Lin(\n, \h)$ is analytic on $\mU$: relative to some bases for $\n$ and $\h$, the entries of its matrix are given by rational functions of $X \in \mU$ \cite{KN}. We also note that as $Z$ is unique, it is $\Ad(H)$-equivariant, which implies that $L$ is a homomorphism of $\h$-modules. In particular, if $\n'$ contains no adjoint submodules, then $L=0$ by Schur's Lemma, and then $[AX,X]=0$, for all $X \in \n'$, by \eqref{eq:LXX} (this implies that all the modules $\m_i \cap \n^\perp$ and $\m_j, \; j \ne i$ pairwise commute).

Furthermore, we have the following useful fact.

\begin{lemma}\label{l:twocoml}
  Suppose the decomposition \eqref{eq:nirr} contains nontrivial modules $\n_1 \ne \n_2$ whose orthogonal complements are both large. Then the metric is naturally reductive.
\end{lemma}
\begin{proof}
  Denote $\q = \oplus_{r=3}^p \n_r$. Let $\mU_r \subset \n_r \oplus \q, \; r=1,2$, be the sets of elements whose centraliser in $\h$ is trivial. Each of the subsets $\mU_r$ is open and dense in $\n_r \oplus \q$. For $r=1,2$, let $\mU'_r$ be the set of those elements $X \in \q$ for which there exists an open and dense subset $\mathcal{V}_{r, X} \subset \n_r$ such that for all $Y_r \in \mathcal{V}_{r, X}$ we have $Y_r + X \in \mU_r$. Note that both $\mU'_1$ and $\mU'_2$ are open and dense in $\q$, as also is the set $\mU'= \mU'_1 \cap \mU'_2$. By Lemma~\ref{l:non+large}, for $r=1,2$, there exist linear maps $L_r: \n_r \oplus \q \to \h$ and maps $\Omega_r: \mU_r \to \Lin(\n_r, \h)$ such that for any $X \in \mU'$ and any $Y_1 \in \mathcal{V}_{1, X}, \, Y_2 \in \mathcal{V}_{2, X}$, the geodesic graph is given by
  \begin{gather}
    Z = L_1(X+Y_2)+\Omega_1(X+Y_2)Y_1 = L_2(X+Y_1)+\Omega_2(X+Y_1)Y_2, \quad\text{and} \label{eq:doubleZ}\\
    \a_1 [L_1(X+Y_2), Y_1] + [\Omega_1(X+Y_2)Y_1,A(X+Y_2)]=[(A-\a_1 \id)(X+Y_2),Y_1], \label{eq:non+l1}
  \end{gather}
  where the latter equations follows from \eqref{eq:non+l} and we assume that $\n_1 \subset \m_1$ and $\n_2 \subset \m_j$ (note that we can have $j=1$). Projecting \eqref{eq:non+l1} to $\q$ we get $[\Omega_1(X+Y_2)Y_1,AX]= \pi_\q[(A-\a_1 \id)X,Y_1]$ (note that $\pi_\q[(A-\a_1 \id)Y_2,Y_1]= (\a_j -\a_1) \pi_\q [Y_2,Y_1]=0$, by Lemma~\ref{l:brackets}~\eqref{it:alphabeta}). Furthermore, by \eqref{eq:OmegaYY} $[\Omega_1(X+Y_2)Y_1,Y_1] = 0$. From the last two equations it follows that for any $X \in \mU', \, Y_1 \in \mathcal{V}_{1, X}$ and any $Y_2', Y_2'' \in \mathcal{V}_{2, X}$ we have $[\Omega_1(X+Y_2')Y_1-\Omega_1(X+Y_2'')Y_1,A(X+Y_1)]=0$. As $X + Y_1 \in \mU_r$ and $\ad_\h$ commutes with $A$ we obtain $\Omega_1(X+Y_2')Y_1 = \Omega_1(X+Y_2'')Y_1$. Therefore there exists a map $\Psi_1: \mU' \to \Lin(\n_1, \h)$ such that $\Omega_1(X+Y_2)Y_1=\Psi_1(X)Y_1$, for all $X \in \mU'$ and all $Y_1 \in \mathcal{V}_{1, X}, \, Y_2 \in \mathcal{V}_{2, X}$. Similarly, there exists a map $\Psi_2: \mU' \to \Lin(\n_2, \h)$ such that $\Omega_2(X+Y_1)Y_2 = \Psi_2(X)Y_2$, for all $X \in \mU'$ and all $Y_1 \in \mathcal{V}_{1, X}, \, Y_2 \in \mathcal{V}_{2, X}$. Substituting into \eqref{eq:doubleZ} we get $Z = L_1X+L_1Y_2+\Psi_1(X)Y_1 = L_2X + L_2Y_1+\Psi_2(X)Y_2$ which now holds for all $X \in \mU'$ and all $Y_1 \in \n_1, \, Y_2 \in \n_2$. Thus $L_1X=L_2X$ and $(\Psi_1(X)-L_2)Y_1 = 0$, for all $X \in \mU'$ and all $Y_1 \in \n_1$. Therefore we have $Z = L_1X+L_1Y_2+L_2Y_1$, for all $X+Y_1+Y_2 \in \m$, and so the metric is naturally reductive by Lemma~\ref{l:linearGG}.
\end{proof}


The following Proposition effectively reduces the list of possible irreducible modules which may appear in the decomposition~\eqref{eq:nirr} to a finite number of candidates, for every given group $H$.

\begin{proposition} \label{p:nolarge}
If one of the irreducible modules in the decomposition \eqref{eq:nirr} is large, then either the metric is naturally reductive or $M$ is $\SU(3)/\SU(2)$ or $\Sp(2)/\Sp(1)$. 
\end{proposition}
\begin{proof}
By Lemma~\ref{l:red} we can assume that no $\m_i$ contains a nonzero ideal of $\g$. Otherwise, that ideal would be a trivial module, and so by factoring it out we would not lose a large module in decomposition~\eqref{eq:nirr}. By Lemma~\ref{l:trivial}\eqref{it:onenont} we can assume that at least two modules in the decomposition \eqref{eq:nirr} are nontrivial (one can easily check that all the irreducible modules in the corresponding decompositions are small except for in the two cases given in the proposition). Furthermore, by Lemma~\ref{l:twocoml}, we can assume that the decomposition \eqref{eq:nirr} has exactly two nontrivial modules, one of them being large and another one, small. Denote by them $\n$ and $\n'$ respectively.

We first suppose $\n'$ is the adjoint module. Then by Lemma~\ref{l:non+large} applied to $\n'$ we obtain that for an open, dense subset $\mU \subset (\n')^\perp \cap \m$, a map $\Omega: \mU \to \Lin(\n', \h)$ and a linear map $L: \mU \to \h$, the geodesic graph is given by $Z=LX + \Omega(X)Y$ for all $X \in \mU, \, Y \in \n'$. Moreover, by \eqref{eq:OmegaYY} we have $[\Omega(X)Y, Y] = 0$, for all $X \in \mU, \, Y \in \n'$. As $\n'$ is the adjoint module, there exists a linear isomorphism $\iota: \h \to \n'$ such that for all $V_1, V_2 \in \h$ we have $[V_1, \iota V_2] = \iota [V_1, V_2]$ (see~\ref{ss:ad}). For $X \in \mU$ define an endomorphism $P_X \in \End(\h)$ by $P_XV=\Omega(X)\iota V$. Then for all $V \in \h$ we have $[P_XV, V]=0$, and so by \cite[Theorem~5.28]{LL}, $P_X$ commutes with all $\ad_V, \; V \in \h$. As the adjoint module is of real type, $P_X =f(X) \id_\h$, for some function $f: \mU \to \br$, so that $\Omega(X) Y = f(X) \iota^{-1} Y$, for all $Y \in \n'$. Choosing $X_1, X_2 \in \mU$ such that $X_1+X_2 \in \mU$ and the intersection of $\Span(X_1, X_2)$ with the trivial submodule of $\m$ is zero (the set of such pairs $(X_1, X_2)$ is open end dense in $\mU \times \mU$) we find from \eqref{eq:non+l} that the function $f$ is locally a constant, say $c \in \br$. Then $Z=LX + c\iota^{-1}Y$ on an open subset of $\m$, hence on the whole $\m$ and so the metric is naturally reductive by Lemma~\ref{l:linearGG}.

We can therefore assume that $\n'$ is a tiny module. As $\m$ contains no adjoint modules, the algebra $\g$ is simple by Lemma~\ref{l:trivial}\eqref{it:simple}. Furthermore, assuming that both $\n$ and $\n'$ lie in the same $\m_i$ and applying Lemma~\ref{l:non+large} to $\n'$ we obtain that on the right-hand side of \eqref{eq:non+l}, the vector $AX-\a_i X$ lies in a trivial submodule of $\m$ orthogonal to $\m_i$, and so $[AX-\a_i X,Y] \in \n'$, by Lemma~\ref{l:brackets}\eqref{it:alphabeta}. Projecting \eqref{eq:non+l} to $(\n')^\perp$ we obtain $[\Omega(X)Y,AX]=0$, and so $\Omega(X)Y=0$, for all $X \in \mU \subset (\n')^\perp$ and all $Y \in \n'$. Then $Z=LX$ which implies that the metric is naturally reductive (note that in fact $Z=L=0$ --- see the comment before Lemma~\ref{l:twocoml}).

We can therefore assume that $\n$ and $\n'$ lie in different eigenspaces of $A$. We have $\m=\n \oplus \n' \oplus \t$, where $\t$ is trivial, and so $\h \oplus \t$ is a subalgebra of $\g$ having exactly two irreducible isotropy modules, $\n$ and $\n'$ (by Lemma~\ref{l:trivial}\eqref{it:actionofVa}). Moreover, by Lemma~\ref{GO-criterion}, the restriction of $A$ to $\n \oplus \n'$ gives a GO metric on the space $G/(H K)$, where $K$ is the centraliser of $H$ in $G$ (the Lie algebra of $K$ is $\t$). That metric is not normal, as $\n$ and $\n'$ lie in different eigenspaces of $A$. Since $G$ is simple, by examining the list in \cite[Theorem~2]{CN2019} we get the following candidates for $M$ (we have omitted the spaces in Lemma~\ref{l:trivial}\eqref{it:onenont}, as for them only one submodule in $\m$ is nontrivial): $\Spin(8)/\G_2, \; \SO(9)/\G_2, \; \SO(2n+1)/\SU(n) \, (n \ge 2)$, $\SO(9)/\Spin(7), \; \SU(n+p)/\SU(n) \, (n \ge 3,\, 2 \le p \le n-1), \; \SU(2n+1)/\Sp(n) \, (n \ge 2)$. All these spaces are perfectly good GO spaces and they appear in the list in Theorem~\ref{th:simple}\eqref{it:thirred}, but from the decompositions given in Table~\ref{t:gosimple} one can see that for each of them, all the nontrivial submodules of $\m$ are tiny (and are listed in Table~\ref{t:tiny}). This contradicts the assumption that $\n$ is large.
\end{proof}

\subsection{Adjoint modules}
\label{ss:ad}

For any adjoint module $\s \subset \m$, there is a well-defined linear bijections $\iota: \h \to \s$ such that for all $U, V \in \h$,
\begin{equation}\label{eq:adjaction}
    [U, \iota V ] = \iota [U, V].
\end{equation}

\begin{lemma} \label{l:adj}
The direct sum of the adjoint module and a nontrivial module is a large module.
\end{lemma}
\begin{proof}
Suppose $\s$ is the adjoint module and $\n$ is an irreducible, nontrivial module. Identifying $\s$ with $\h$ via $\iota$ as in \eqref{eq:adjaction} we see that it is sufficient to find two elements $X \in \h, \; Y \in \n$ whose centralisers have trivial intersection. This latter condition means that the rank of the linear system $[Z,X]=[Z,Y]=0$ in the variable $Z$ is maximal (equals $\dim \h$); the set of pairs $(X,Y)$ for which it is not is Zariski closed in the complexification $\h^\bc \times \n^\bc$, and so it is sufficient to construct $X \in \h^\bc, \; Y \in \n^\bc$ whose centralisers in $\h^\bc$ have trivial intersection. To do that, take $X$ to be regular and denote $\cs \subset \h^\bc$ the Cartan subalgebra defined by $X$. Let $\gamma$ be the dominant weight of $\n^\bc$. Then every element of its orbit under the action of the Weyl group $\mathcal{W}$ of $\h^\bc$ on $\cs^*$ is also a weight of $\n^\bc$. 
Furthermore, the orbit $\mathcal{W}(\gamma)$ spans $\cs^*$ as $\h^\bc$ is simple. 
Take $Y$ to be a linear combination of nonzero vectors $Y_g \in V_{g(\gamma)}$, for all $g \in \mathcal{W}$, where $V_{g(\gamma)}$ is the weight space corresponding to the root $g(\gamma)$. Now the centraliser of $X$ is $\cs$, but no nonzero vector from $\cs$ centralises $Y$.
\end{proof}

Furthermore, we have the following proposition.

\begin{proposition} \label{p:noadj}
  If one of the irreducible modules in the decomposition \eqref{eq:nirr} is adjoint, then the metric is naturally reductive.
\end{proposition}
\begin{proof}
Suppose the decomposition \eqref{eq:nirr} contains an adjoint module $\s$. In the assumption that the metric is GO but not naturally reductive, by Lemma~\ref{l:adj}, Lemma~\ref{l:twocoml} and Lemma~\ref{l:trivial}\eqref{it:onenont} we can assume that exactly one other module $\n$ in \eqref{eq:nirr} is nontrivial, so that $\m=\s \oplus \n \oplus \t$, where $\t$ is trivial. By Proposition~\ref{p:nolarge} we can assume that such $\n$ is small. Furthermore, we can assume that $\m$ contains no simple ideals of $\g$. For if $\g_a \subset \m$ is a simple ideal of $\g$, then $\h$ lies in the sum of other ideals of $\g$, and so $\g_a$ is a trivial $\h$-module. But then by Lemma~\ref{l:trivial}\eqref{it:Vaideals} it entirely lies in one of the eigenspaces $\m_i$ and we can factor it out by Lemma~\ref{l:red}.

We first assume that both $\s$ and $\n$ lie in the same eigenspace $\m_1$ of $A$. Then $\m_2$ is a nonzero, trivial module. Take a nonzero $T \in \m_2$. As $\s$ is of real type we have $[T, \s]=0$. Take $X=S+N+T$, where $S \in \s, \, N \in \n$. By \eqref{eq:GO}, there exists $Z \in \h$ such that
\begin{equation*}
  0= [Z+X, AX] = \a_1 [Z,S] + \a_1 [Z,N] + (\a_2-\a_1)[N,T],
\end{equation*}
As $[T, \n] \subset \n$ by Lemma~\ref{l:trivial}\eqref{it:actionofVa} the latter equation gives $[Z,S] = 0$ and $[Z,N]=(1-\a_1^{-1}\a_2)[T,N]$. Let $\cs \subset \h$ be a Cartan subalgebra and let $V \in \cs$ be a regular vector. Taking $S=\iota V$ we obtain that $Z \in \cs$ by \eqref{eq:adjaction}. By Lemma~\ref{l:trivial}\eqref{it:noideal}, \eqref{it:nonz} we can assume, up to scaling, that the restriction of $(1-\a_1^{-1}\a_2)\ad_T$ to $\n$ is an almost Hermitian structure, and then by Lemma~\ref{l:trivial}\eqref{it:actionofVa}, the restriction of $\ad_\h$ to $\n$ is a subalgebra of $\su(\n)$ (it lies in $\ug(\n)$, the centraliser of $(\ad_T)_{|\n}$ and hence in $\su(\n)$ as $\h$ is simple). Then $(\ad_\cs)_{|\n}$ is an abelian subalgebra of $\su(\n)$ which lies in a Cartan subalgebra $\cs'$ of $\su(\n)$. But then choosing a unitary basis for $\n$ we find that the equation $[Z,N]=(1-\a_1^{-1}\a_2)[T,N]$ is equivalent to the fact that for $x \in \bc^n \; (2n = \dim \n)$, there is a real, diagonal $n \times n$ matrix $D$ with $\Tr D = 0$ such that $\mathrm{i}Dx=\mathrm{i}x$, which is false for a generic $x \in \bc^n$, a contradiction.

Now suppose $\s$ and $\n$ lie in different eigenspaces of $A$. The homogeneous space $\hat{M}=G/(H K)$ where $K$ is the connected subgroup of $G$ whose Lie algebra is $\t$ has exactly two irreducible components in its isotropy representations (note that $\t$ acts separately on $\s$ and on $\n$ by Lemma~\ref{l:trivial}\eqref{it:actionofVa}) and moreover, the restriction of $A$ to $\s \oplus \n$ defines a GO metric on $\hat{M}$ which is not normal. By \cite[Proposition~1]{CN2019} we can have one of three cases (note that $\g$ must be semisimple and no ideal of it is allowed to be orthogonal to $\h$). In the first case, $\g=\h \oplus \h \oplus \h$ and $\h \subset \g$ is the diagonal (then $\t=0$). Then $M$ is a Ledger-Obata space and the metric is naturally reductive by \cite[Proposition~3]{CN2019} (see also \cite[Proposition~1]{NNLO}; in fact, any invariant metric on the Ledger-Obata space $H^3/H$ is naturally reductive, even without imposing the GO condition). In the second case, the algebra $\g$ is simple. Examining the cases in \cite[Theorem~2]{CN2019}, we find that in neither of them $\m$ contains an adjoint module. In the third case, we have $\g=\g_1 \oplus \g_2$, where $\g_1, \g_2$ are simple ideals, with both projections $\h_1$ and $\h_2$ from $\h$ to $\g_1$ and to $\g_2$ respectively being isomorphic to $\h$, and with $\g_1 = \h_1$. Then $\g_2=\h_2 \oplus \n \oplus \t$, the algebra $\h \subset \g$ is the diagonal in $\h_1 \oplus \h_2$, and the adjoint module $\s$ is its orthogonal complement in $\h_1 \oplus \h_2$. Note that then $[\s, \n] \subset \n$, and moreover, the action of $\s$ on $\n$ coincides with that of $\h$, that is, $[S,N] = [\iota^{-1}S, N]$, for $S \in \s, \, N \in \n$ (note that $\iota$ is defined up to scaling and we can take it to be a linear isometry). Up to relabelling, we have $\s \subset \m_1, \, \n \subset \m_2$. Suppose $T \in \t_j$, the trivial submodule of $\m_j, \, j=1,2, \dots, m$. Then \eqref{eq:GO} with $X=S+N+T, \, S \in \s, \, N \in \n$, gives that there exists $Z \in \h$ such that $\a_1[Z,S] + \a_2 [Z,N] + (\a_2-\a_1) [S,N] + (\a_2-\a_j) [T, N] = 0$ (we used the fact that $[T, \s]=0$). Set $V = \iota^{-1}S \in \h$. Then $[S,N]=[V,N]$, and so the GO condition is equivalent to
\begin{equation}\label{eq:adjsemi}
  [Z, V] = 0, \quad [\a_2 Z + (\a_2-\a_1)V, N] = (\a_j-\a_2) [T, N].
\end{equation}
If $(\a_j-\a_2) T \ne 0$, we argue as in the previous paragraph: by Lemma~\ref{l:trivial}\eqref{it:noideal}, \eqref{it:nonz}, $(\a_j-\a_2) (\ad_T)_{|\n}$ is a nonzero multiple of an almost Hermitian structure on $\n$, and then for a regular $V \in \h$, from the first equation of \eqref{eq:adjsemi}, $Z$ lies in the Cartan subalgebra defined by $V$ which lies in a Cartan subalgebra of $\su(\n)$. But then the second equation of \eqref{eq:adjsemi} cannot be satisfied with a generic $N \in \n$, a contradiction. It now follows from Lemma~\ref{l:trivial}\eqref{it:Vaideals} that $\t =\t_2 \subset \m_2$, and then \eqref{eq:adjsemi} is satisfied with $Z(S+N+T)= (\a_1\a_2^{-1} - 1)V = (\a_1\a_2^{-1} - 1) \iota^{-1}S$. So the metric is naturally reductive by Lemma~\ref{l:linearGG}.
\end{proof}

\section{$G$-GO spaces. Tiny modules}
\label{s:tiny}


Now we are in a position to complete the proof of Theorem~\ref{th:simple}. 

In the assumptions of Theorem~\ref{th:simple} we assume that \emph{the GO metric is not naturally reductive}. Summarising the results of the previous sections we can additionally assume the following:
\begin{itemize}
  \item
  all nontrivial modules in the decomposition \eqref{eq:nirr} are tiny, and there are at least two of them (by Propositions~\ref{p:nolarge} and \ref{p:noadj} and Lemma~\ref{l:trivial}\eqref{it:onenont});

  \item
  there are not ``too many" of them: there is no more than one nontrivial module whose complement is large (by Lemma~\ref{l:twocoml});

  \item
  no $\m_i$ contains an ideal of $\g$ by Lemma~\ref{l:red};

  \item
  $\g$ is simple (by Lemma~\ref{l:trivial}\eqref{it:simple});

  \item
  and finally, note that no GO metric constructed below is naturally reductive, unless it is normal (by Propositions~\ref{p:noadj} and Remark~\ref{rem:LO}).
\end{itemize}

The list of tiny modules from \cite[Table~1]{HH} is given in Table~\ref{t:tiny}. Note that some simple groups (e.g. $\SU(2)$ and $\E_8$) have no tiny representations, while some others may have up to three. In the second column, for representations coming from the $s$-representations of compact symmetric spaces, we give those spaces. The fourth column indicates the type: real, complex or quaternionic; the fifth, the principal stationary subgroup.  

\begin{table}
\renewcommand{\arraystretch}{1.2}
  \centering
  \begin{tabular}{|l|l|c|c|c|}
  \hline
  Group & Representation & dim & type & Stationary \\
  \hline
  $\SO(n), \; n \ge 5$ & standard, $\br^n$ & $n$ & r & $\SO(n-1)$ \\
  $\SU(n), \; n \ge 3$ & standard, $\bc^n$ & $2n$ & c & $\SU(n-1)$ \\
  $\Sp(n), \; n \ge 2$ & standard, $\bH^n$ & $4n$ & q & $\Sp(n-1)$ \\
  $\SU(n), \; n \ge 5$ & $s: \SO(2n)/\mathrm{U}(n)$ & $n(n-1)$ & c  & $\SU(2)^{[\frac{n}2]}$ \\ 
  $\Sp(n), \; n \ge 3$ & $s: \SU(2n)/\Sp(n)$ & $(n-1)(2n+1)$ & r & $\Sp(1)^n$ \\ 
  $\SU(6)$ & $s: \E_6/\SU(6) \SU(2)$ & 40 & q & $T^2$ \\ 
  $\Spin(7)$ & spin & $8$ & r & $\G_2$ \\ 
  $\Spin(9)$ & spin, $s: \F_4/\Spin(9)$ & 16 & r & $\Spin(7)$ \\ 
  $\Spin(10)$ & spin, $s: \E_6/\Spin(10) \SO(2)$ & 32 & c & $\SU(4)$ \\ 
  $\Spin(12)$ & spin, $s: \E_7/\Spin(12) \SU(2)$ & 64 & q & $\SU(2)^3$ \\ 
  $\G_2$ & standard, $\Oc \cap 1^\perp$ & $7$ & r & $\SU(3)$ \\ 
  $\F_4$ & $s: \E_6/F_4$ & $26$ & r & $\Spin(8)$ \\ 
  $\E_6$ & $s: \E_7/\E_6 \SO(2)$ & $54$ & c & $\Spin(8)$ \\ 
  $\E_7$ & $s: \E_8/\E_7 \SU(2)$ & $112$ & q & $\Spin(8)$ \\ 
  \hline
\end{tabular}
\vskip 0.2cm
  \caption{Tiny modules}\label{t:tiny}
\end{table}

In the rest of the proof, we consider the groups in Table~\ref{t:tiny} one-by-one. Our strategy, for every individual group, will be first to consider all the possible decompositions \eqref{eq:nirr}; there will be a finite number of them: the nontrivial submodules are controlled by the above assumptions, and the trivial ones, by Lemma~\ref{l:trivial}. Some of those cases will be then sorted out by the dimension count, as $\g$ must be simple. The remaining ones, when there are only two nontrivial modules, can be reduced to the classification in \cite[Theorem~2]{CN2019} (in particular, if the trivial submodule $\t \subset \m$ is zero, this classification applies directly). For the small number of remaining cases, we consider possible ``distributions" of the modules in the decomposition~\eqref{eq:nirr} among the eigenspaces $\m_i, \; i=1, \dots, m$, of the metric endomorphism $A$ (note that $m \ge 2$) using Lemma~\ref{l:brackets}, the decompositions of the tensor products and the external squares into irreducible modules and the classification of compact irreducible symmetric spaces. If no contradiction is reached up to this point, we apply the GO criterion from Lemma~\ref{GO-criterion} to determine the GO metric; then we identify the corresponding space from the list in Theorem~\ref{th:simple}\eqref{it:thirred} (and in Table~\ref{t:gosimple}).

Throughout this section we use the notation introduced earlier (in Sections~\ref{s:pre} and \ref{ss:trivial}); the direct sum of $a \ge 0$ copies of a module $\n$ is abbreviated to $a \n$. 

\subsection{Types B and D: $\mathbf{H=SO(n), Spin(n), \; n \ge 5}$}
\label{ss:so(n)}

\subsubsection{\underline{$\SO(n), \; n \ge 5$}}
\label{sss:sonstand}
From Table~\ref{t:tiny}, there is only one tiny module, $\br^n$, the standard one. It follows that in the decomposition \eqref{eq:nirr}, all the modules are either standard or trivial. Let $\m_1 = a_1 \br^n \oplus \t_1, \; \m_2 = a_2 \br^n \oplus \t_2$ be two eigenspaces of $A$, where $a_1, a_2 \ge 0$ and $\t_1, \t_2$ are trivial. We have $[\t_1, \t_2]=0$ by Lemma~\ref{l:trivial}\eqref{it:Vaideals}, and then for any $\n_1 = \br^n \subset \m_1, \; \n_2 = \br^n \subset \m_2$, we have $[\t_1, \n_2]=[\t_2, \n_1]=0$ by Lemma~\ref{l:trivial}\eqref{it:actionofVa} and \eqref{it:realS} and $[\n_1, \n_2]=0$ as the irreducible decomposition of $\br^n \otimes \br^n$ contains no module $\br^n$. It follows that $[\m_1, \m_2]=0$. Therefore all the modules $\m_i$ pairwise commute, and so the metric is naturally reductive (we can take $Z=0$ in \eqref{eq:GO}).

\subsubsection{\underline{$\Spin(7)$}}
\label{sss:spin7}
From Table~\ref{t:tiny}, there are only two tiny modules, the $8$-dimensional module $\s$ for the spin representation and the $7$-dimensional module $\br^7$ for the standard representation of $\SO(7)$. If there are no spin modules, we have a representation of $\SO(7)$ and then any GO metric must be naturally reductive as we have shown above. We therefore assume that there is at least one spin module. We claim that the sum of any four modules each of which is either spin or standard (and at least one is spin) is large. Indeed, if we have four spin modules, then choosing a generic element in one of them we get the stationary subgroup $\G_2$ represented on $\br^8$ as the automorphism group of the algebra of octonions $\Oc$. As any three non-associating octonions generate $\Oc$ we obtain that the stationary subgroup of a generic quadruple of elements is trivial. Next, suppose we have one standard module $\br^7$ and three spin modules. Then the stationary subgroup of a nonzero element from $\br^7$ is $\Spin(6)=\SU(4)$, and its representation on each of the spin modules is the standard representation of $\SU(4)$ on $\bc^4=\br^8$. The stationary subgroup of a generic triple of elements is trivial. Next, suppose we have two standard $\br^7$ modules and two spin modules. Then the stationary subgroup of a generic pair of elements from $\br^7$ is $\Spin(5)=\Sp(2)$, and its representation on each of the spin modules is the standard representation of $\Sp(2)$ on $\bH^2=\br^8$. The stationary subgroup of a generic pair of elements is again trivial. Finally, if we have three standard $\br^7$ modules and one spin module, the stationary subgroup of a generic triple of elements from $\br^7$ is $\Spin(4)=\Sp(1) \times \Sp(1)$, and its representation on the spin module is the sum of the two standard representations of $\Sp(1)$ on $\bH=\br^4$. The stationary subgroup of a generic element in $\br^8$ (the sum of two elements from each copy of $\bH$) is again trivial.

Up to relabelling, the decomposition \eqref{eq:nirr} takes the form $\m=\oplus_{r=1}^q \n_r \oplus \t$, where $\t$ is a trivial module, and from among the modules $\n_r$, for $r=1, \dots, q$, we have $s \ge 1$ spin modules and $a \ge 0$ standard $\br^7$ modules, with $s+a=q$. By Lemma~\ref{l:twocoml} and the arguments above we can assume that $q \le 4$, and by Lemma~\ref{l:trivial}\eqref{it:onenont}, that $q \ge 2$. We will consider several cases depending on the value of $q \in \{2,3,4\}$ and the ``distribution" of nontrivial modules among the eigenspaces $\m_i$. For $i=1, \dots, m$, we have $\m_i= s_i \s \oplus a_i \br^7 \oplus \t_i$, where $\t_i$ are trivial modules. We have $s=\sum_{i=1}^{m} s_i, \; a=\sum_{i=1}^{m} a_i$, with $s \ge 1$ and $2 \le q (=s+a) \le 4$. Note that $m > 1$ (otherwise the metric is normal). As both the spin module and the standard module are of real type, Lemma~\ref{l:trivial}\eqref{it:Vaideals}, \eqref{it:actionofVa} implies that each $\t_i$ commutes with all $\m_j, \; j \ne i$, and may not commute with $\m'_i$ only when $\m'_i$ contains at least two isomorphic modules. Then by Lemma~\ref{l:trivial}\eqref{it:noideal} we obtain that for no $i=1, \dots, m$, the module $\m_i$ can be trivial (so that $a_i+s_i > 1$, for all $i=1, \dots, m$), and that $\t_i$ can only be nonzero when either $s_i > 1$ or $a_i > 1$, and in that case, $\t_i$ is isomorphic to a subalgebra of $\so(s_i) \oplus \so(a_i)$.

The above argument shows that if $q=2$, then $\m$ contains no trivial submodules (for if both nontrivial submodules lie in the same $\m_1$, then $\m=\m_1$). Then by the result of \cite[Theorem~2]{CN2019}, we get the GO space $M=\Spin(9)/\Spin(7)$ in Theorem~\ref{th:simple}\eqref{it:thirred}\textbf{\eqref{it:97}}, with $\m=\m_1 \oplus \m_2$, where $\m_1$ is the spin module and $\m_2$ is the standard module. 

Assume that $q=3$ or $q=4$. Furthermore, we have the following irreducible decompositions of $\Spin(7)$ modules:
\begin{equation} \label{eq:spin7otimes}
\s \otimes \s = \br \oplus \br^7 \oplus \so(7) \oplus \dots, \quad \br^7 \otimes \br^7 = \br \oplus \so(7) \oplus \dots, \quad \s \otimes \br^7 = \s \oplus \dots,
\end{equation}
where $\so(7)$ is the adjoint module and the dots denote the sums of irreducible large modules (these modules cannot occur in the decomposition of $\g$ viewed as the $\Spin(7)$ module). It then follows from Lemma~\ref{l:brackets}\eqref{it:alphabeta} that any two spin modules lying in different eigenspaces $\m_i, \m_j, \; i \ne j$, commute. Therefore if $\m$ contains no standard submodules $\br^7$, then any two different eigenspaces $\m_i$ commute and hence the metric is normal. We can therefore further assume that $a \ge 1$.

As $\g$ must be simple, from the dimension count and the above conditions we obtain the following list of candidates (where $\s_j$ are spin modules).

\begin{enumerate} [label=(\roman*),ref=\roman*]
  \item \label{it:spin75}
  $q=4, \; \g=\fg_4$, and $\n_1=\s_1, \; \n_2 = \s_2, \; \n_3=\s_3, \; \n_4= \br^7$.
  \item \label{it:spin73}
  $q=4, \; \g=\fg_4$, and $\m_1=\s_1 \oplus \s_2 \oplus \t_1, \; \n_3=\br^7, \; \n_4= \br^7$, where $\t_1=\so(2)$.
  \item \label{it:spin74}
  $q=4, \; \g=\fg_4$, and $\m_1=\br^7 \oplus \br^7 \oplus \t_1, \; \n_3=\s_1, \; \n_4= \s_2$, where $\t_1=\so(2)$.
  \item \label{it:spin72}
  $q=4, \; \g=\so(11)$ or $\g=\sp(5)$, and $\m_1=\s_1 \oplus \s_2 \oplus \s_3 \oplus \t_1, \; \m_2=\br^7$, where $\t_1=\so(3)$.
  \item \label{it:spin71}
  $q=3, \; \g=\so(10)$, and $\m_1=\s_1 \oplus \s_2 \oplus \t_1, \; \m_2=\br^7$, where $\t_1=\so(2)$.
\end{enumerate}

Note that in cases (\ref{it:spin75}, \ref{it:spin73}) and \eqref{it:spin74} the irreducible submodules $\n_r$ may lie either in the same or in different eigenspaces $\m_i$.

By Lemma~\ref{l:brackets}\eqref{it:alphabeta}, \eqref{it:twoinone} and from \eqref{eq:spin7otimes} we find that the sum $\h'$ of $\h=\so(7)$ and all the trivial and all the standard submodules of $\m$ is a subalgebra of $\g$ (not necessarily simple) and that its orthogonal complement $\p$ (which is the sum of all the spin submodules of $\m$) satisfies $[\p,\p] \subset \h'$. It follows that $(\g, \h')$ is a symmetric pair, and additionally $\dim(\g/\h')$ is a multiple of $8$.

In particular, if $\g=\fg_4$, the classification in \cite{Hel} shows that there is only one such pair, $(\g,\h')=(\fg_4, \so(9))$, which corresponds to the Cayley projective plane. We immediately see that case \eqref{it:spin75} is not possible by the dimension count. Case \eqref{it:spin73} is also not possible, because from Lemma~\ref{l:trivial}\eqref{it:actionofVa}, $\t_1$ would lie in the centre of $\h'=\so(9)$ which is trivial. In case \eqref{it:spin74} we have $\h'=\h \oplus \br^7 \oplus \br^7 \oplus \t_1 = \so(9)$ (note that $\t_1$ acts nontrivially on $\br^7 \oplus \br^7$). Moreover, the modules $\s_1$ and $\s_2$ commute --- this follows from Lemma~\ref{l:brackets}\eqref{it:alphabeta} and \eqref{eq:spin7otimes} if they lie in two different eigenspaces of the metric automorphism $A$, and from Lemma~\ref{l:brackets}\eqref{it:twoinone} and \eqref{eq:spin7otimes} if they lie in the same eigenspace. But this is a contradiction as no two linear independent vectors in the tangent space of the Cayley projective plane may commute (as elements of $\fg_4$), since otherwise the sectional curvature of the two-plane spanned by them would equal zero.

We now separately consider two remaining cases.

\eqref{it:spin72} In this case, $\h \oplus \br^7 = \so(8)$ and $\t_1$ lies in the centre of $\h'$. There is no symmetric pair $(\g,\h')=(\sp(5), \so(8) \oplus \so(3))$, and so $\g=\so(11)$ giving the symmetric pair $(\g,\h')=(\so(11), \so(8) \oplus \so(3))$. The corresponding homogeneous space is $\SO(11)/\Spin(7)$, where $\Spin(7) \subset \SO(8) \subset \SO(8) \times \SO(3)$. It is an $S^7$-fibration over the Stieffel manifold $\SO(11)/\SO(8)$ with a normal metric (the construction is similar to the that in \cite[Section~2]{T1}, but with the non-symmetric base).

By Lemma~\eqref{GO-criterion}, the GO condition is equivalent to the fact that for any $X_r \in \s_r, \, r=1,2,3$, and $T \in \t_1, \, Y \in \br^7$, there exists $Z \in \h$ such that $0=[X_1+X_2+X_3+T+Y+Z, \a_1(X_1+X_2+X_3+T)+\a_2 Y]=(\a_1-\a_2)[Y,X_1+X_2+X_3+T] + \a_2 [Z,Y] + \a_1 [Z,X_1+X_2+X_3+T] = \a_1 [Z+ (1-\a_2\a_1^{-1}) Y, X_1] + \a_1 [Z+ (1-\a_2\a_1^{-1}) Y, X_2] + \a_1 [Z+ (1-\a_2\a_1^{-1}) Y, X_3] +\a_2 [Z,Y]$. By \eqref{eq:spin7otimes} and Lemma~\ref{l:brackets}\eqref{it:alphabeta}, the four terms on the right-hand side belong to $\s_1, \s_2, \s_3$ and $\br^7$ respectively, and so the GO condition is equivalent to the existence of $Z \in \h$ such that
\begin{equation}\label{eq:spin7go}
[Z,Y]=0, \qquad [Z+ (1-\a_2\a_1^{-1}) Y, X_r]=0, \quad \text{for } r=1, 2, 3.
\end{equation}
Now if $Y=0$, one can take $Z=0$. If $Y \ne 0$, then from the first equation, $Z$ belongs to the stationary subalgebra $\so(6)=\su(4) \subset \so(7)$ of $Y$. Identifying three modules $\s_1, \s_2, \s_3$ with a single spin module $\s$ via an isomorphism we see that relative to some unitary basis for $\s=\bc^4$, the operator $(\ad_Y)_{|\s}$ is proportional to the multiplication by $\mathrm{i}$, and the action of $\su(4)$ commutes with it. Then \eqref{eq:spin7go} is equivalent to $[Z, X_r]=\mu \mathrm{i} X_r, \; r=1,2,3$, where $\mu \in \br, \, \mu \ne 0$. It is sufficient to show that a required $Z \in \su(4)$ exists for $X_1, X_2, X_3$ unitary orthonormal. Extending $\{X_1, X_2, X_3\}$ to a unitary basis $\{X_1, X_2, X_3, X_4\}$ we can take $Z \in \su(4)$ such that $(\ad_Z)_{|\s}=\diag(\mu \mathrm{i}, \mu \mathrm{i}, \mu \mathrm{i}, -3\mu \mathrm{i})$ relative to this basis. It follows that the metric is GO; it is not naturally reductive unless it is normal. This gives space \eqref{it:117} in Theorem~\ref{th:simple}\eqref{it:thirred}.

\eqref{it:spin71} In this case, $\h \oplus \br^7 = \so(8)$ and $\t_1$ lies in the centre of $\h'$. We obtain the symmetric pair $(\g,\h')=(\so(10), \so(8) \oplus \so(2))$. The corresponding homogeneous space is $\SO(10)/\Spin(7)$, where $\Spin(7) \subset \SO(8) \subset \SO(8) \times \SO(2)$, which is an $S^7$-fibration over the Stieffel manifold $\SO(10)/\SO(8)$ with a normal metric. The fact that the metric is GO, can be established by repeating the arguments for the previous case, with obvious modifications. This is the space \eqref{it:107} in Theorem~\ref{th:simple}\eqref{it:thirred}.

\subsubsection{\underline{$\Spin(9)$}}
\label{sss:spin9}
From Table~\ref{t:tiny}, there are only two tiny modules, the $16$-dimensional module $\s$ for the spin representation and the $9$-dimensional module $\br^{9}$ for the standard representation of $\SO(9)$. Then in the decomposition \eqref{eq:nirr}, all the modules $\n_r$ are either spin or standard or trivial. We have the following irreducible decompositions of $\Spin(9)$ modules:
\begin{equation} \label{eq:spin9otimes}
\begin{gathered}
\s \otimes \s = \br \oplus \br^9 \oplus \so(9) \oplus \dots, \quad \Lambda^2 \s = \so(9) \oplus \dots, \\
\br^9 \otimes \br^9 = \br \oplus \so(9) \oplus \dots, \quad \s \otimes \br^9 = \s \oplus \dots,
\end{gathered}
\end{equation}
where $\so(9)$ is the adjoint module and dots denote the sums of irreducible large modules (these modules cannot occur in the decomposition of $\g$ viewed as the $\Spin(9)$ module). Let $\m_1$ and $\m_2$ be two eigenspaces of the metric endomorphism $A$ and let $\t_i, \, i=1,2$, be the (maximal) trivial submodule of $\m_i$. Suppose $\s_i \subset \m_i, \, i=1,2$, are spin submodules and $(\br^9)_i \subset \m_i, \, i=1,2$, are standard submodules. Then by Lemma~\ref{l:brackets}\eqref{it:alphabeta} and \eqref{eq:spin9otimes} we have $[\s_1,\s_2]=0$ and $[(\br^9)_1,(\br^9)_2]=0$, and by Lemma~\ref{l:trivial}\eqref{it:Vaideals}, $[\t_1,\t_2]=0$. Furthermore, as both the spin and the standard modules are of real type, Lemma~\ref{l:trivial}\eqref{it:actionofVa} gives $[\t_1,\s_2 \oplus (\br^9)_2]=[\t_2,\s_1 \oplus (\br^9)_1]=0$. Finally, from Lemma~\ref{l:brackets}\eqref{it:alphabeta} and \eqref{eq:spin9otimes} we get $[(\br^9)_1,\s_2] \subset \s_2$, which gives a homomorphism from $(\br^9)_1$ to $\Lambda^2 \s_2$. But this must be trivial by \eqref{eq:spin9otimes}, which implies that $[(\br^9)_1,\s_2]=0$ and similarly, $[(\br^9)_2,\s_1]=0$. It follows that $[\m_1, \m_2]=0$. Therefore all the modules $\m_i$ pairwise commute, and so the metric is naturally reductive.

\subsubsection{\underline{$\Spin(10)$}}
\label{sss:spin10}
From Table~\ref{t:tiny}, there are only two tiny modules, the $32$-dimensional module $\s$ for the spin representation and the $10$-dimensional module $\br^{10}$ for the standard representation of $\SO(10)$. If $\m$ contains no spin modules, we can take $H = \SO(10)$ and then any GO metric must be naturally reductive by \ref{sss:sonstand}. We can therefore assume that $\m$ contains at least one spin module, and moreover, at least one other nontrivial module, by Lemma~\ref{l:trivial}\eqref{it:onenont}.

We first show that the sum of two spin modules is large. The spin representation comes from the $s$-representation: it is the representation of $\Spin(10)$ on the tangent space of the symmetric space $Q=\E_6/\Spin(10) \SO(2)$. This symmetric space has rank $2$, with the restricted root system of type $\mathrm{BC}_2$; there are $6$ roots: $\ve_1$ and $\ve_2$ of multiplicity $8$ each, $\ve_1 \pm \ve_2$ of multiplicity $6$ each and $2\ve_1$ and $2\ve_2$ of multiplicity $1$ each \cite[Table~1]{T2}. The stationary subalgebra $\k(0)$ of a regular element of a maximal abelian subalgebra $\ag \subset \q = T_oQ$ is $\su(4) \oplus \so(2)$. By \cite[Lemma~2.25(a)]{N}, for every $6$-dimensional root space $\q(\la), \; \la=\ve_1 \pm \ve_2$, the subalgebra spanned by $[\q(\la),\q(\la)]$ is an ideal of $\k(0)$ isomorphic to $\so(6)$, and it acts as the standard representation of $\so(6)$ on $\q(\la)$. By \cite[Corollary~2.26(a)]{N}, for every $8$-dimensional root space $\q(\la), \; \la=\ve_1, \ve_2$, the subalgebra spanned by $[\q(\la),\q(\la)]$ is an ideal of $\k(0)$ isomorphic to $\su(4)$, and it acts as the standard representation of $\su(4)$ on $\q(\la)$. It is easy to see that the stationary subalgebra in $\k(0)$ of the element $X_++X_-+Y_1+Y_2$, where $X_\pm \in \q(\ve_1 \pm \ve_2)$ and $Y_r \in \q(\ve_r), \, r=1,2$, are generic vectors, is trivial. Therefore the sum of two copies of $\s$ is a large module. 

We next show that the sum of $\s$ and three copies of $\br^{10}$ is large. Indeed, the stationary subgroup of a triple of linear independent elements of $\br^{10}$ is $\Spin(7) \subset \Spin(10)$. Then $\s$ is the sum of four $8$-dimensional irreducible $\Spin(7)$ submodules, and so is large by the argument in \ref{sss:spin7}.

We can therefore assume by Lemma~\ref{l:twocoml} that the decomposition \eqref{eq:nirr} takes either the form $\m=\s_1 \oplus \s_2 \oplus \t$, where $\t$ is a trivial module, or the form $\m=\s \oplus a \br^{10} \oplus \t$, where $\t$ is a trivial module and $1 \le a \le 3$.

In the first case, as $\s$ is of complex type, we find by Lemma~\ref{l:trivial} that $\t$ is isomorphic to either a subalgebra of $\mathfrak{u}(2)$ if $\s_1$ and $\s_2$ lie in the same eigenspace $\m_i$, or of $\mathfrak{u}(1) \oplus \mathfrak{u}(1)$, if they lie in different eigenspaces. The dimension count shows that, in any case, $109 \le \dim \g \le 113$, but there are no simple Lie algebras of such dimensions, a contradiction.

In the second case, if $a=1$, as $\s$ is of complex type and $\br^{10}$ is of real type, we obtain by Lemma~\ref{l:trivial} that $\t$ is isomorphic to a subalgebra of $\mathfrak{u}(1)$. Then $\dim \g \in \{87, 88\}$. Similarly, if $a=3$, we obtain that $\t$ is isomorphic to a subalgebra of $\mathfrak{u}(1) \oplus \so(3)$, and so $107 \le \dim \g \le 111$. But in both cases, there are no simple Lie algebras of such dimensions. For $a=2$, $\t$ is isomorphic to a subalgebra of $\mathfrak{u}(1) \oplus \so(2)$, and from the dimension count, the only candidate for $\g$ is $\su(10)$. But $\Spin(10)$ cannot be a subgroup of $\SU(10)$ as $\Spin(10)$ has no faithful real representation on $\br^{20}$.

So in the case $H=\Spin(10)$, any GO metric is naturally reductive.

\subsubsection{\underline{$\Spin(12)$}}
\label{sss:spin12}
From Table~\ref{t:tiny}, there are only two tiny modules, the $64$-dimensional spin module $\s$ and the $12$-dimensional module $\br^{12}$ for the standard representation of $\SO(12)$. We can assume that $\m$ contains at least one spin module by~\ref{sss:sonstand}, and at least one other nontrivial module by Lemma~\ref{l:trivial}\eqref{it:onenont}. We claim that the sum of the spin module and either of the spin or the standard module is large. Indeed, if the second module is the standard module $\br^{12}$, the stationary subgroup of a nonzero element of it is $\Spin(11) \subset \Spin(12)$. Its representation on $\s$ is still irreducible, with the trivial principal stationary subgroup (no $\Spin(11)$ entry in Table~\ref{t:tiny}). Now suppose the second module is also spin. The spin representation comes from the $s$-representation for the symmetric space $Q=\E_7/\Spin(12) \SU(2)$. The symmetric space $Q$ has rank $4$, with the restricted root system of type $\F_4$; there are $12$ roots of multiplicity $1$ and $12$ roots of multiplicity $4$ \cite[Table~1]{T2}. The stationary subalgebra $\k(0)$ of a regular element of a maximal abelian subalgebra $\ag \subset \q = T_oQ$ is $\so(3) \oplus \so(3) \oplus \so(3)$. By \cite[Lemma~2.25(a)]{N}, for every $4$-dimensional root space $\q(\la)$, the subalgebra spanned by $[\q(\la),\q(\la)]$ is an ideal of $\k(0)$ isomorphic to $\so(4)$, and it acts as the standard representation of $\so(4)$ on $\q(\la)$. Take three $4$-dimensional root spaces $\q(\la)$ such that $[\q(\la), \q(\la)]$ span the same ideal $\so(4) \subset \k(0)$ and choose three generic vectors, $X_1,X_2,X_3$, one in each of them. Then the stationary subalgebra of $X_1+X_2+X_3$ in $\k(0)$ is the ``remaining" ideal $\so(3)$. Now choose a nonzero vector $Y$ in a $4$-dimensional root space $\q(\la)$ such that $[\q(\la),\q(\la)]$ contains that ideal. Then the stationary subalgebra of $X_1+X_2+X_3+Y$ in $\k(0)$ is trivial. Therefore the sum of two copies of $\s$ is a large module. 

We can therefore assume that the decomposition \eqref{eq:nirr} takes either the form $\m=\s_1 \oplus \s_2 \oplus \t$, where $\t$ is a trivial module, or the form $\m=\s \oplus \br^{12} \oplus \t$, where $\t$ is a trivial module.

In the second case, both $\s$ and $\br^{12}$ are $\ad(\t)$-invariant by Lemma~\ref{l:trivial}\eqref{it:actionofVa}. As $\s$ is of quaternionic type and $\br^{12}$ is of real type, $\t$ is a subalgebra of $\sp(1)$, by Lemma~\ref{l:trivial}\eqref{it:noideal},\eqref{it:realS}. It follows that $142 \le \dim \g \le 145$. The only simple Lie algebra $\g$ whose dimension lies in this range is $\su(12)$, but $\Spin(12)$ cannot be a subgroup of $\SU(12)$ as $\Spin(12)$ has no faithful real representation on $\br^{24}$. Similarly, in the first case, $\t$ must be a subalgebra of $\sp(2)$, which gives $194 \le \dim \g \le 204$. The only simple Lie algebra $\g$ whose dimension lies in this range is $\su(15)$, but $\Spin(12)$ is not a subgroup of $\SU(15)$ as $\Spin(12)$ has no faithful real representation on $\br^{30}$.

So in the case $H=\Spin(12)$, any GO metric is naturally reductive.

\subsection{Type A: $\mathbf{H=SU(n), \; n \ge 3}$}
\label{ss:su(n)}

From Table~\ref{t:tiny}, we can have the following tiny $\SU(n)$ modules. For all $n \ge 3$, we have the standard module of dimension $2n$, and for all $n \ge 5$, we have the module $\p$ of dimension $n(n-1)$ coming from the $s$-representation for the symmetric space $Q=\SO(2n)/\mathrm{U}(n)$ (note that for $n=3$ this module is standard, and for $n=4$ it is reducible). In addition, for $n=4$ we have the tiny module of dimension $6$ coming from the standard representation of $\SO(6)=\SU(4)/\mathbb{Z}_2$, and for $n=6$, there is a tiny module $\q$ of dimension $40$ from the $s$-representation for the symmetric space $Q=\E_6/\SU(6) \SU(2)$. 

Before considering various cases we prove the following lemma. Consider the homogeneous space $G/H=\SO(2n+1)/\SU(n)$, where $n \ge 3$ and $H=\SU(n) \subset \mathrm{U}(n) \subset \SO(2n)\subset \SO(2n+1)=G$. At the level of Lie algebras, we have $\h=\su(n) \subset \ug(n) \subset \so(2n) \subset \so(2n+1)=\g$. We have an orthogonal decomposition into $\h$-modules: $\g= \h \oplus \t \oplus \n \oplus \s$, where $\t$ is the one-dimensional trivial module which is the orthogonal complement to $\h$ in $\ug(n)$, the module $\n$ is the orthogonal complement to $\ug(n)$ in $\so(2n)$ and the module $\s$ is the standard $2n$-dimensional module which is the orthogonal complement to $\so(2n)$ in $\so(2n+1)$ (note that the module $\n$ is reducible when $n=4$ and is standard when $n=3$).

\begin{lemma} \label{l:so2n+1sun}
In the above notation, suppose the metric on $\SO(2n+1)/\SU(n), \; n \ge 3$, is defined by a metric automorphism $A$ such that the modules $\s, \n, \t$ lie in the eigenspaces of $A$ with the eigenvalues $\a_1, \a_2, \a_3 > 0$ respectively. Then if $n$ is even, the metric is GO if and only if $\a_2=\a_3$; if $n$ is odd, the metric is GO if and only if $n \alpha_3^{-1} = (n-1) \alpha_2^{-1} + \alpha_1^{-1}$.
\end{lemma}
\begin{proof}
We have $[\n, \s]=[\t, \s]=\s$ and $[\t, \n]=\n$. For $X \in \s, \, Y \in \n, \, T \in \t$, the GO condition \eqref{eq:GO} gives $[Z+X+Y+T, \a_1 X + \a_2 Y + \a_3 T]=0$, which is equivalent to
\begin{equation*}
  [Z + \sigma T, Y]=0, \qquad [Z + \rho Y + \tau T, X]=0,
\end{equation*}
where $\sigma = 1-\a_3\a_2^{-1}, \; \rho = 1-\a_2\a_1^{-1}, \; \tau = 1-\a_3\a_1^{-1}$. Choose an almost Hermitian structure in $\s=\bc^n$ in such a way that $(\ad_t)_{|\s_1}$ is a real multiple of the multiplication by $\ic$. Then for $w = X \in \bc^n$ we have $[Z, X]= Mw$, where $M$ is an $n \times n$ skew-Hermitian matrix with trace zero, $[T, X]= \ic t w$, where $t \in \br$, and $[Y, X]= NCw$, where $N$ is an $n \times n$ complex skew-symmetric matrix (depending on $Y$) and $C$ is the componentwise complex conjugation. The above GO condition is then equivalent to the following: for any $w \in \bc^n, \, t \in \br$ and $N \in \so(n, \bc)$, there exists an $n \times n$ traceless, skew-Hermitian matrix $M$ such that
\begin{equation} \label{eq:sungo}
  (M +\sigma t \ic I_n)N + N(M +\sigma t \ic I_n)^t =0, \qquad (M+\rho NC+ \tau t \ic)w=0.
\end{equation}

We consider two cases depending on the parity of $n$.

\smallskip

Suppose $n$ is even; let $n=2k$. Take $N$ in \eqref{eq:sungo} to be generic (that is, $\rk N = n$ and all the eigenvalues of $N$ are pairwise distinct). We can choose a unitary basis for $\bc^n$ relative to which $N=\diag(\mu_1 J, \mu_2 J, \dots, \mu_k J)$, where $J= \left(\begin{smallmatrix} 0 & 1 \\ -1 & 0 \end{smallmatrix}\right)$ and $\pm\mu_1, \pm\mu_2, \dots, \pm\mu_k \in \bc$ are nonzero and pairwise distinct. Then from the first equation of~\eqref{eq:sungo} we obtain $M + \sigma t \ic I_n = \diag(F_1, F_2, \dots, F_k)$, where $F_1, F_2, \dots, F_k \in \su(2)$. It follows that $\Tr (M + \sigma t \ic I_n) = 0$ and so $\sigma = 0$. This gives $\a_2=\a_3$ and $\tau = \rho$.

Now the second equation of~\eqref{eq:sungo} gives
\begin{equation*}
  F_j w^j + \rho \mu_j J \overline{w}^j + \rho t \ic w^j = 0, \qquad j=1, 2, \dots, k,
\end{equation*}
where the coordinates of the vector $w^j = \left(\begin{smallmatrix} w^j_1 \\ w^j_2 \end{smallmatrix}\right) \in \bc^2$ are the $(2j-1)$-st and the $2j$-th coordinates of $w$ respectively. If $w^j=0$, the latter equation is trivially satisfied, with an arbitrary $F_j$. Otherwise, a direct calculation gives $F_j = \left(\begin{smallmatrix} a_j \ic & z_j \\ -\overline{z_j} & -a_j \ic \end{smallmatrix}\right)$, where $a_j \in \br$ and $z_j \in \bc$ are given by
\begin{equation*}
  (|w^j_1|^2 + |w^j_2|^2) \begin{pmatrix} a_j \ic \\ z_j \end{pmatrix} = - \rho t \ic \begin{pmatrix} |w^j_1|^2 - |w^j_2|^2 \\ 2w^j_1 \overline{w}^j_2 \end{pmatrix} + \rho \begin{pmatrix} 2 \ic \Im(\mu_j \overline{w}^j_1\overline{w}^j_2) \\ \mu_j (\overline{w}^j_2)^2+ \overline{\mu_j} (w^j_1)^2 \end{pmatrix} \, .
\end{equation*}
As the right-hand side is continuous in $\mu_j$, the entries of $N$, we deduce that a traceless skew-Hermitian matrix $M$ which satisfies \eqref{eq:sungo} exists for all $N \in \so(n, \bc)$ (and all $t \in \br$ and $w \in \bc^n$). Hence the metric so defined is GO.

\smallskip

Now suppose $n$ is odd; let $n=2k+1$. The proof is similar to that in the even case with some modifications. We again take a generic $N$ and choose a unitary basis for $\bc^n$ relative to which $N=\diag(\mu_1 J, \mu_2 J, \dots, \mu_k J, 0)$, where $J= \left(\begin{smallmatrix} 0 & 1 \\ -1 & 0 \end{smallmatrix}\right)$ and $\pm\mu_1, \pm\mu_2, \dots, \pm\mu_k \in \bc$ are nonzero and pairwise distinct. The first equation of~\eqref{eq:sungo} gives $M + \sigma t \ic I_n = \diag(F_1, F_2, \dots, F_k, c)$, where $F_1, F_2, \dots, F_k \in \su(2)$ and $c \in \bc$. Comparing the traces we find that $c= n \sigma t \ic$ and so $M + \sigma t \ic I_n = \diag(F_1-\sigma t \ic I_2, F_2-\sigma t \ic I_2,$ $\dots, F_k-\sigma t \ic I_2, 2k \sigma t \ic)$.

Then the second equation of~\eqref{eq:sungo} gives
\begin{equation*}
  (F_j + (\tau - \sigma) t \ic I_2) w^j + \rho \mu_j J \overline{w}^j = 0, \qquad j=1, 2, \dots, k, \qquad 2k \sigma + \tau = 0,
\end{equation*}
where the coordinates of the vector $w^j = \left(\begin{smallmatrix} w^j_1 \\ w^j_2 \end{smallmatrix}\right) \in \bc^2$ are the $(2j-1)$-st and the $2j$-th coordinates of $w$ respectively. From the last equation we obtain $n \alpha_3^{-1} = (n-1) \alpha_2^{-1} + \alpha_1^{-1}$. From the first $k$ equations we find, provided $w^j \ne 0$, that the entries of the matrix $F_j = \left(\begin{smallmatrix} a_j \ic & z_j \\ -\overline{z_j} & -a_j \ic \end{smallmatrix}\right) \in \su(2), \; a_j \in \br, \, z_j \in \bc$ are given by
\begin{equation*}
  (|w^j_1|^2 + |w^j_2|^2) \begin{pmatrix} a_j \ic \\ z_j \end{pmatrix} = n \sigma t \ic \begin{pmatrix} |w^j_1|^2 - |w^j_2|^2 \\ 2w^j_1 \overline{w}^j_2 \end{pmatrix} + \rho \begin{pmatrix} 2 \ic \Im(\mu_j \overline{w}^j_1 \overline{w}^j_2) \\ \mu_j (\overline{w}^j_2)^2+ \overline{\mu_j} (w^j_1)^2 \end{pmatrix} \, .
\end{equation*}
Similar to the even case this proves that the metric is GO.
\end{proof}

\subsubsection{\underline{$\SU(n), \; n \ge 4$; all modules standard}}
\label{sss:sunallst}
We first consider the case when $n \ge 4$ and all nontrivial submodules in the decomposition \eqref{eq:nirr} are standard. There has to be no more than $n-1$ of them, as the sum of $n-1$ standard modules is already a large module.

Note that the tensor square of the standard module contains no standard modules, and the exterior square of the standard module is the sum of the adjoint module, the module $\p$ defined above and the one-dimensional trivial module. Suppose we have two standard modules $\s_1 \subset \m_i$ and $\s_2 \subset \m_j, \; i \ne j$. Then by Lemma~\ref{l:brackets}\eqref{it:alphabeta} $[\s_1, \s_2] \subset \s_1 \oplus \s_2$, and hence $\s_1$ and $\s_2$ commute; in particular, $[[\s_1, \s_1],\s_2]=0$. On the other hand, $[\s_1, \s_1]$ is contained in the subalgebra which is the direct sum of $\h$ and a one-dimensional trivial module, and moreover, is an ideal in that subalgebra. As $\s_1$ is nontrivial we obtain $\h \subset [\s_1, \s_1]$ which contradicts the fact that $[[\s_1, \s_1],\s_2]=0$, as $\s_2$ is also nontrivial. It follows that all the standard modules $\s_r$ lie in the same eigenspace, and so up to relabelling we have $\m_1=\oplus_{r=1}^p \s_r \oplus \t_1$ and $\m_i=\t_i$ for $i =2, \dots, m$, where $\t_1, \t_2, \dots, \t_m$ are trivial modules. Note that $2 \le p \le n-1$ (if $p=1$ we obtain the space $\SU(n + 1)/\SU(n)$ from Lemma~\ref{l:trivial}\eqref{it:onenont}) and $m \ge 2$, so that $\t_2 \ne 0$. Moreover, the module $\t_1$ must also be nonzero, for if $\t_1=0$ we obtain that $[\s_1, \s_2]=0$ by Lemma~\ref{l:brackets}\eqref{it:twoinone} and then repeat the argument above. Furthermore, the pair $(\g, \h')$, where $\h'=\h \oplus \t$, is a symmetric pair (as the bracket of standard modules contains no standard modules and hence lies in $\h'$, and $\m_1':=\oplus_{r=1}^p \s_r$ is $\ad(\t)$-invariant, by Lemma~\ref{l:trivial}\eqref{it:actionofVa}). As $\g$ is simple and $\h'$ is the sum of two ideals, $\su(n)$ and $\t$,  with $\dim \t \ge 2$, from the classification \cite{Hel} we find that $(\g, \h') = (\su(n+p), \su(n) \oplus \su(p) \oplus \br)$. From Lemma~\ref{l:trivial}\eqref{it:noideal}, \eqref{it:Vaideals} we obtain that $m=2$ and $\t_1 \oplus \t_2 = \su(p) \oplus \br$. Note that $\t_2 \ne \su(p)$ as $\ad_{\t_2}$ preserves every individual module $\s_r, \; r=1, \dots, p$, by  Lemma~\ref{l:trivial}\eqref{it:actionofVa}; therefore $\t_1=\su(p), \, \t_2 = \br$. It follows that $\m_1=\oplus_{r=1}^p \s_r \oplus \su(p)$ and $\m_2=\br$. The corresponding homogeneous space is $\SU(n+p)/\SU(n), \; 2 \le p \le n-1$, where $\SU(n) \subset \SU(n) \times \SU(p) \subset \mathrm{S}(\mathrm{U}(n) \times \mathrm{U}(p))$; this is the space \eqref{it:susu} in Theorem~\ref{th:simple}\eqref{it:thirred} (with $p \ge 2$ and $n \ge 4$).

By Lemma~\eqref{GO-criterion}, the GO condition is equivalent to the fact that for any $X_r \in \s_r, \, r=1, \dots, p$, and $T_1 \in \t_1, \, T \in \t_2$, there exists $Z \in \h$ such that $[\sum_{r=1}^p X_r+T_1+T+Z$, $\a_1(\sum_{r=1}^p X_r+T_1)+\a_2 T]=0$. Note that $\t_2$ commutes with both $\h$ and $\t_1$, and preserves each of the $\s_r$'s. Therefore the GO condition is equivalent to the existence of $Z \in \h$ such that
\begin{equation}\label{eq:sunstgo}
[Z+ (1-\a_2\a_1^{-1}) T, X_r]=0, \quad \text{for } r=1, \dots, p.
\end{equation}
If $T=0$, we can take $Z=0$. If $T \ne 0$, then identifying the modules $\s_r$ with a single standard module $\n$ via an isomorphism we see that relative to some unitary basis for $\n=\bc^n$, the operator $(\ad_T)_{|\s}$ is real proportional to the multiplication by $\mathrm{i}$, and the action of $\su(n)$ commutes with it. Then \eqref{eq:sunstgo} is equivalent to $[Z, X_r]=\mu \mathrm{i} X_r, \; r=1,\dots, p$, where $\mu \in \br, \, \mu \ne 0$. It is sufficient to show that a required $Z \in \su(n)$ exists when $X_r$ are unitary orthonormal. Extending $\{X_1, \dots, X_p\}$ to a unitary basis $\{X_1, \dots, X_p, \dots, X_n\}$ (recall that $p \le n-1$) we can take $Z \in \su(n)$ such that $(\ad_Z)_{|\s}$ relative to this basis is given by the diagonal matrix whose first $p$ entries are $\mu \mathrm{i}$ and the remaining $n-p$, are $-p(n-p)^{-1}\mu \mathrm{i}$. It follows that the metric is GO.

\smallskip

\subsubsection{\underline{$\SU(n), \; n \ge 5$ and $n \ne 6$}}
\label{sss:sunallstp}
From Table~\ref{t:tiny}, there are two tiny modules, the standard module $\s = \bc^n$ and the $n(n-1)$-dimensional module $\p$, and so in the decomposition \eqref{eq:nirr}, every nontrivial submodule is isomorphic to either $\s$ or $\p$. There are at least two of them, and we can assume that at least one is isomorphic to $\p$ by \ref{sss:sunallst}. Note that both $\p \oplus \p$ and $\p \oplus \s$ are large modules. Indeed, for a generic element of $\p$ we have
the set of $[n/2]$ mutually unitary orthogonal subspaces $\bc^2 \subset \bc^n$ such that the stationary subalgebra is the sum of $[n/2]$ copies of $\su(2)$ each acting in its own $\bc^2$. To see that the principal stationary subalgebra of $\p \oplus \s$ is trivial we take an element in $\s=\bc^n$ all whose components in the subspaces $\bc^2$ are nonzero; for $\p \oplus \p$, we take two generic elements of $\p$ sharing no common subspaces $\bc^2$ in the corresponding sets.

It follows that the only possible cases for the decomposition \eqref{eq:nirr} are $\m = \p \oplus \p \oplus \t$ and $\m = \p \oplus \s \oplus \t$, where $\t$ is a trivial module.

We will need the irreducible decompositions of the tensor products of $\s$ and $\p$:
\begin{equation} \label{eq:sunotimes}
\begin{gathered}
\s \otimes \s = 2 \br \oplus \p \oplus 2 \su(n) \oplus \dots, \quad \text{for } n = 5 \text{ or } n \ne 7, \\
\begin{aligned}
&\p \otimes \p = 2 \br \oplus 2 \su(n) \oplus \dots, & \quad  &\s \otimes \p = \s \oplus \dots, & \quad &\text{for } n \ge 7, \\
&\p \otimes \p = 2 \br \oplus 2 \su(5) \oplus \s \oplus \dots, & \quad & \s \otimes \p = \s \oplus \p \oplus \dots, & \quad &\text{for } n = 5,
\end{aligned}
\end{gathered}
\end{equation}
where $\br$ is the trivial module, $\su(n)$ is the adjoint module and dots denote the sums of irreducible large modules.

We now consider the case $\m = \p \oplus \p \oplus \t$. From Lemma~\ref{l:trivial} we find $\t \subset \ug(2)$. Furthermore, by \eqref{eq:sunotimes}, $\p \otimes \p$ does not contain $\p$, and so $(\g, \h' = \su(n) \oplus \t)$ is a symmetric pair, with the corresponding symmetric space of dimension $2n(n-1)$. From the classification in \cite{Hel} we find that there are no such pairs for $n \ge 5$, and so this case is not possible.

Next consider the case $\m = \p \oplus \s \oplus \t$. From Lemma~\ref{l:trivial} we have $\t \subset \br^2$. First suppose that $n \ge 7$. Then by \eqref{eq:sunotimes}, $(\g, \h' = \su(n) \oplus \p \oplus \t)$ is a symmetric pair, with the corresponding symmetric space of dimension $2n$ and with $2n^2-n-1 \le \dim \h' \le 2n^2-n+1$ and $\rk \h' \ge n-1$. From the classification in \cite{Hel} we find that there is only one such pair: $(\g,\h') = (\so(2n+1), \so(2n))$. Then we get $\dim \t = 1$ and $\su(n) \oplus \p \oplus \t = \so(2n)$ and we obtain a family of GO metrics as in Lemma~\ref{l:so2n+1sun}.

The last case to consider is $\m = \p \oplus \s \oplus \t$ and $n = 5$. We again have $\dim \t \le 2$ and then from the dimension count, $\dim \g \in \{54, 55, 56\}$. Then we get $\dim \t = 1$ and $\g=\sp(5)$ or $\g=\so(11)$. The first case is not possible, as the only way to realise the algebra $\ug(5)=\su(5) \oplus \t$ as a subalgebra of $\sp(5)$ is the one corresponding to the symmetric pair $(\sp(5), \ug(5))$. But as both $\p$ and $\s$ are $\ad(\t)$-invariant, the $30$-dimensional irreducible $\ug(5)$-module which is the tangent space of the corresponding symmetric space remains irreducible for the subalgebra $\su(5)$. In the case $\g=\so(11)$ there is again, the only way to realise $\ug(5)=\su(5) \oplus \t$ as a subalgebra of $\so(11)$: we have $\ug(5) \subset \so(10) \subset \so(11)$, and then this case is completed by application of Lemma~\ref{l:so2n+1sun}.

Thus we obtain the spaces in Theorem~\ref{th:simple}\eqref{it:thirred}\eqref{it:so2su} with $n \ge 7$ and $n=5$.

\subsubsection{\underline{$\SU(3)$}}
\label{sss:su3}
From Table~\ref{t:tiny}, there is only one tiny module, the standard one, of dimension $6$. Note that the sum of two of them is a large module, and so by Lemma~\ref{l:twocoml} and Lemma~\ref{l:trivial}\eqref{it:onenont} we can assume that the decomposition \eqref{eq:nirr} takes the form $\m=\s_1 \oplus \s_2 \oplus \t$, where $\s_1$ and $\s_2$ are standard and $\t$ is a trivial module. As the standard module is of complex type and as there are two of them, from Lemma~\ref{l:trivial} we find that $\t$ is a subalgebra of $\ug(2)$. It follows that $20 \le \dim \g \le 24$. This gives three possibilities: either $\g=\sp(3), \; \dim \t =1$, or $\g=\so(7), \; \dim \t = 1$, or $\g=\su(5), \; \dim \t = 4$.

In the first case, there is only one way, up to conjugation, to realise the algebra $\ug(3)=\h \oplus \t$ as a subalgebra of $\g=\sp(3)$, namely the one corresponding to the symmetric pair $(\sp(3), \ug(3))$. As we have at least two different eigenspaces $\m_i$ of $A$, and $\dim \t =1$, we obtain by Lemma~\ref{l:trivial}\eqref{it:actionofVa} that both $\s_1$ and $\s_2$ are $\ad(\t)$-invariant. But then they are $\ug(3)$-invariant which contradicts the fact that the representation of $\ug(3)$ on the tangent space of the corresponding symmetric space is irreducible.

In the second case, there is again only one way, up to conjugation, to realise the algebra $\ug(3)=\h \oplus \t$ as a subalgebra of $\g=\so(7)$: we have $\ug(3) \subset \so(6) \subset \so(7)$, where the first inclusion corresponds to the symmetric pair $(\so(6), \ug(3))$. We have therefore the following modules in the decomposition \eqref{eq:nirr}: the standard module $\s_1$ which is the orthogonal complement to $\so(6)$ in $\so(7)$, the standard module $\s_2$ which is the orthogonal complement to $\ug(3)$ in $\so(6)$, and the one-dimensional module $\t$, the centre of $\ug(3)$. By Lemma~\ref{l:so2n+1sun} we get GO metrics on the space $\SO(7)/\SU(3)$ from Theorem~\ref{th:simple}\eqref{it:thirred}\eqref{it:so2su} with $n = 3$.

In the third case, from Lemma~\ref{l:trivial} we obtain that $\t$ can have dimension $4$ only if, up to relabelling, the eigenspaces of $A$ are given by $\m_1=\s_1 \oplus \s_2 \oplus \t_1, \; \m_2 = \t_2$, where $\t_1$ is isomorphic to $\su(2)$ and $\dim \t_2 = 1$. Repeating the arguments in the last two paragraphs of \ref{sss:sunallst}, we obtain a family of GO metrics on $\SU(5)/\SU(3)$ (the space \eqref{it:susu} in Theorem~\ref{th:simple}\eqref{it:thirred} with $n = 3$ and $p=2$).

\subsubsection{\underline{$\SU(4)$}}
\label{sss:su4}
From Table~\ref{t:tiny}, there are two tiny modules, the standard module $\s = \bc^4$ and the module $\n=\br^6$ coming from the standard representation of $\SO(6)=\SU(4)/\mathbb{Z}_2$ (note that the orthogonal complement to $\ug(4)$ in $\so(8)$ is reducible and is the direct sum of two copies of $\n$).

The cases when all nontrivial modules in the decomposition \eqref{eq:nirr} are isomorphic to $\n$ or all are standard have already been considered in \ref{sss:sonstand} and \ref{sss:sunallst} respectively. We can therefore assume that at least one module in \eqref{eq:nirr} is standard and at least one is isomorphic to $\n$. Note that the module $2\n \oplus \s$ is large. Indeed, the stationary subgroup of a nonzero element from $\s = \bc^4$ is $\SU(3)$, and the restriction of the representation of $\SU(4)$ on $\n$ to $\SU(3)$ is the standard representation of $\SU(3)$ on $\n=\br^6=\bc^3$. Then the stationary subgroup of a pair of linear independent vectors from $\bc^3$ is trivial. Similarly, the module $\n \oplus 2\s$ is large, as the stationary subgroup of a nonzero element from $\n$ is $\Sp(2)=\Spin(5) \subset \Spin(6)=\SU(4)$, and so $\s$ can be viewed as the standard module for $\Sp(2)$. Then the stationary subgroup of a pair of linear independent vectors from $\s = \bH^2$ is trivial.

It therefore follows that the decomposition \eqref{eq:nirr} takes the form $\m = q \s \oplus r \n \oplus \t$, where $\t$ is a trivial module and $(q,r) = (1,1), (2,1), (1,2)$. Note that $\s$ is of complex type and $\n$ is of real type. Then in the case $(q,r) = (1,1)$, by Lemma~\ref{l:trivial}\eqref{it:actionofVa}, \eqref{it:noideal} and \eqref{it:realS} we obtain $\dim \t \le 1$, and so $\dim \g \in \{29, 30\}$, which is a contradiction as there are no simple groups of these dimensions. Similarly, in the case $(q,r) = (2,1)$ we obtain that $\t$ is a subalgebra of $\ug(2)$, and so $37 \le \dim \g \le 41$, which again leads to a contradiction. In the case $(q,r) = (1,2)$, Lemma~\ref{l:trivial} gives $\dim \t \le 2$, and from the dimension count we obtain one of the following three cases: either $\t=0$ and $\g=\su(6)$, or $\dim \t = 1$ and then $\g=\sp(4)$ or $\g=\so(9)$.

We have the following irreducible decompositions of $\SU(4)$ modules:
\begin{equation} \label{eq:su4otimes}
\s \otimes \s = 2 \br \oplus 2 \n \oplus 2 \su(4) \oplus \dots, \quad \n \otimes \n = \br \oplus \su(4) \oplus \dots, \quad \s \otimes \n = \s \oplus \dots,
\end{equation}
where $\su(4)$ is the adjoint module and dots denote the sums of irreducible large modules. It follows that $\h'=\h \oplus 2 \n \oplus \t$ is a subalgebra of $\g$ and that $(\g, \h')$ is a symmetric pair. The corresponding symmetric space must be of dimension $8$, and from the classification in \cite{Hel} we find that the only possible case is $\g=\so(9), \; \h'=\so(8)$, and then $\s$ is the orthogonal complement to $\so(8)$ is $\so(9)$, $\t$ is the one-dimensional centraliser of $\su(4)$ in $\so(8)$ and $2 \n$ is the orthogonal complement to $\ug(4)=\su(4) \oplus \t$ in $\so(8)$. Denote by $\n_1, \n_2$ the two copies of $\n$. If $\n_1$ and $\n_2$ lie in different eigenspaces $\m_i$, then by Lemma~\ref{l:brackets}\eqref{it:alphabeta} and \eqref{eq:su4otimes} we get $[\n_1, \n_2]=0$, and so $[[\n_1, \n_1], \n_2]=0$ which contradicts the fact that both $\n_1$ and $\n_2$ are irreducible nontrivial modules. Therefore $\n_1 \oplus \n_2$ lie in the same eigenspace $\m_i$ of $A$. By a similar argument using Lemma~\ref{l:brackets}\eqref{it:twoinone} we obtain that $\t \subset \m_i$. It follows that the metric endomorphism $A$ has two eigenspaces, $\m_1=\n_1 \oplus \n_2 \oplus \t$ and $\m_2 = \s$. By Lemma~\ref{l:so2n+1sun}, we get the space $\SO(9)/\SU(4)$ in Theorem~\ref{th:simple}\eqref{it:thirred}\eqref{it:so2su} with $n = 4$.

\subsubsection{\underline{$\SU(6)$}}
\label{sss:su6}
We have three tiny modules, the standard module $\s$ of dimension $12$, the module $\p$ of dimension $30$ and the module $\q$ of dimension $40$ which comes from the $s$-representation for the symmetric space $Q=\E_6/\SU(6) \SU(2)$. As above, we can assume that in the decomposition~\eqref{eq:nirr} we have at least two nontrivial modules and that at least one of them is not standard. We claim that the sum of any two modules each of which is isomorphic to $\s, \p$ or $\q$ and at least one is not standard is a large module. For the modules $\p \oplus \p$ and $\p \oplus \s$ the argument is similar to that in the first paragraph of \ref{sss:sunallstp}. In the remaining cases, one of the modules is $\q$ with the stationary subalgebra $\br^2$ (see Table~\ref{t:tiny}) which is a subalgebra of a Cartan subalgebra of $\su(6)$. It follows that the sum of $\q$ and any nontrivial module is large, by the argument in the proof of Lemma~\ref{l:adj}.

Therefore the decomposition~\eqref{eq:nirr} takes the form $\m = \n_1 \oplus \n_2 \oplus \t$, where $\n_1, \n_2 \in \{\s, \p, \q\}$ and at least one of them is not standard, and $\t$ is trivial. Note that $\s$ and $\p$ are of complex type and $\q$ of quaternionic type. We consider all possible cases.

Suppose $\m = 2\q \oplus \t$. Then by Lemma~\ref{l:trivial} we have $\t \subset \sp(2)$, and by the dimension count we get $\dim \g = 115 + \dim \t$, so that $115 \le \dim \g \le 125$. The only two simple algebras of such dimensions are $\su(11)$ and $\so(16)$, both of dimension $120$. But $\sp(2)$ contains no subalgebras of dimension $5$, so this case is not possible.

Suppose $\m = \p \oplus \q \oplus \t$. Then $\t \subset \sp(1) \oplus \br$, and $\dim \g = 105 + \dim \t$, so that $105 \le \dim \g \le 109$. The only two simple algebras of such dimensions are $\sp(7)$ and $\so(15)$, both of dimension $105$. But then $\t=0$ and so $\m = \p \oplus \q$ which implies that any GO metric (if one at all exists) is normal, by \cite[Theorem~2]{CN2019}.

Suppose $\m = \s \oplus \q \oplus \t$. Then again $\t \subset \sp(1) \oplus \br$, and so $\dim \g = 87 + \dim \t$ which gives the only candidate $\so(14)$. Then $\t = \sp(1) \oplus \br$. But the only faithful representation of $\su(6)$ of dimension at most $14$ is the standard representation on $\br^{12}$. It follows that $\su(6) \subset \so(12) \subset \so(14)$, and so the centraliser of $\su(6)$ in $\so(14)$ is $\br^2$ (cannot be as large as $\t=\sp(1) \oplus \br$). This contradiction shows that this case is also not possible.

Next suppose $\m=2\p \oplus \t$. By Lemma~\ref{l:trivial} we have $\t \subset \ug(2)$, and by the dimension count, the only possible case is $\g=\su(10), \; \t = \ug(2)$. But the only realisation of $\su(6)$ as a subalgebra of $\su(10)$ comes from the inclusion $\bc^6 \subset \bc^{10}$, and then the centraliser of $\su(6)$ is $\ug(4)$ which is much bigger than $\t$. So this case is also not possible.

Finally let $\m=\p \oplus \s \oplus \t$. Then $\t \subset \br^2$ and the dimension count gives $\dim \t =1$ and $\dim \g = 78$, so that $\g$ is either $\eg_6$ or $\sp(6)$ or $\so(13)$. But the case $\g=\eg_6$ is not possible, as the only $\su(6)$ subalgebra in $\eg_6$ is the one coming from the symmetric pair $(\eg_6, \su(6) \oplus \su(2))$, and its centraliser contains $\su(2)$ in contradiction with $\dim \t =1$. Furthermore, if $\g=\sp(6)$, then the only $\su(6)$ subalgebra is the one coming from the symmetric pair $(\sp(6), \ug(6))$. But then, as both $\p$ and $\s$ are $\ad(\t)$-invariant, they are also $\ug(6)$-invariant which contradicts the fact that the representation of $\ug(6)$ on the tangent space of the corresponding symmetric space is irreducible.
The last remaining case is $\g=\so(13)$. The only way to realise $\ug(6)=\su(6) \oplus \t$ as a subalgebra of $\so(13)$ is $\ug(6) \subset \so(12) \subset \so(13)$, where $(\so(12), \ug(6))$ is a symmetric pair. By Lemma~\ref{l:so2n+1sun} we obtain the space $\SO(13)/\SU(6)$ in Theorem~\ref{th:simple}\eqref{it:thirred}\eqref{it:so2su} with $n = 6$.

\subsection{Type C: $\mathbf{H=Sp(n), \; n \ge 2}$} 
\label{ss:sp(n)}

From Table~\ref{t:tiny}, we can have two tiny $\Sp(n)$ modules: the standard module $\s$ of dimension $4n$ and the module $\p$ of dimension $(n-1)(2n+1)$ obtained from the $s$-representation for the symmetric space $Q=\SU(2n)/\Sp(n)$ (note that for $n=2$, the module $\p$ comes from the standard representation of $\SO(5)=\Sp(2)/\mathbb{Z}_2$ on $\br^5$).

We first consider the case when all the nontrivial modules in decomposition \eqref{eq:nirr} are standard. The sum of $n$ such modules is a large module, and so we can assume that $\m = \oplus_{i=1}^p \s_i \oplus \t$, where $\s_i$ are standard modules, $\t$ is trivial and $2 \le p \le n$. As the tensor square of the standard module contains no standard modules, we find that $(\g, \h'=\h \oplus \t)$ is a symmetric pair, with the corresponding symmetric space of dimension $4pn$, and with $\t$ being a nonzero ideal in $\h'$ (for if $\t=0$, the eigenspaces $\m_i$ are the sums of the standard modules and hence pairwise commute by Lemma~\ref{l:brackets}\eqref{it:alphabeta}; this gives a naturally reductive metric: we can take $Z=0$ in \eqref{eq:GO}). From the classification in \cite{Hel}, there are only two cases. In the first one, we have $n=2$ and $(\g, \h')= (\so(5+q), \so(5) \oplus \so(q))$; but then $8p=5q$, from the dimension count, and so $p \ge 5$, a contradiction. In the second case, $n \ge 2$ is arbitrary and $(\g, \h')= (\sp(n+q), \sp(n) \oplus \sp(q))$. It follows that $q=p$ and that $\t=\sp(p)$, and so by Lemma~\ref{l:trivial}\eqref{it:Vaideals}, $\t$ entirely lies in a single eigenspace $\m_i$. But then from Lemma~\ref{l:trivial}\eqref{it:actionofVa}, we can have $\t$ so big only when all the standard modules $\s_1, \dots, \s_p$ also lie in the same eigenspace which implies that $A$ is the multiple of the identity and so the GO metric is normal (note that this argument does not work when $p=1$, in which case we obtain the space $\Sp(n+1)/\Sp(n)$ from Lemma~\ref{l:trivial}\eqref{it:onenont} which carries a GO metric which is not naturally reductive).

We can therefore assume that the decomposition \eqref{eq:nirr} contains the module $\p$ and at least one other nontrivial module. Note that $\p \oplus \s$ is a large module. Indeed, the stationary subalgebra of a generic element of $\p$ is $n \, \sp(1)$, which acts on $\s=\bH^n$ by acting as $\sp(1)$ on $n$ orthogonal copies of $\bH$; taking an element of $\s$ whose component in each of these copies is nonzero we find that the stationary subalgebra is trivial. Assume that all nontrivial modules in \eqref{eq:nirr} are isomorphic to $\p$. If $n=2$, we have the standard representation of $\so(5)$ on $\br^5$, and so the sum of four copies of $\p$ is large, but the sum of three is still small. If $n \ge 3$, the sum of two copies of $\p$ is already large. To see this, we identify the module $\p$ with the tangent space $\q$ of the symmetric space $Q=\SU(2n)/\Sp(n)$ and choose a generic element $X \in \q$. Its centraliser in $\q$ is a maximal abelian subalgebra $\ag \subset \q$ of dimension $n-1$ which then defines the decomposition of $\q \cap \ag^\perp$ into the orthogonal sum of $\frac12 n(n-1)$ root spaces $\q(\la)$ of dimension $4$, with the root system of type $A_{n-1}$ \cite[Table~1]{T2}. The stationary subalgebra of $X$ is $\k(0) = n \, \sp(1) = \oplus_{i=1}^{n} (\sp(1))_i$, with every root space $\q(\la)$ being $\k(0)$-invariant, and moreover, for the root $\la=\pm(\ve_i-\ve_j), \; 1 \le i < j \le n$, the subalgebra $(\sp(1))_i \oplus (\sp(1))_j \subset \k(0)$ acts on $\q(\la)$ as the standard representation of $\so(4)$, and all the other $(\sp(1))_k, \, k \ne i, j$, act on $\q(\la)$ trivially \cite[Lemma~2.25]{N}. It follows that if we choose $Y \in \q$ whose component in each of the root spaces $\q(\la)$ is nonzero, its stationary subalgebra in $\k(0)$ will be trivial.

Suppose all the nontrivial modules in the decomposition \eqref{eq:nirr} are isomorphic to $\p$. If $n=2$ we have the standard representation of $\so(5)$ on $\br^5$; this case has been analysed in~\ref{sss:sonstand}. If $n \ge 3$ then from the above, $\m$ contains exactly two copies of $\p$. As $\p$ is of real type, no module $\m_i$ can be trivial by Lemma~\ref{l:trivial}\eqref{it:noideal}, \eqref{it:realS} and \eqref{it:actionofVa}, and so the two copies of $\p$ must lie in different eigenspaces $\m_i$, and moreover, neither of these eigenspaces contains a trivial submodule. It follows that the only possibility is $\m =\m_1 \oplus \m_2$, with $\m_1=\p_1, \, \m_2 = \p_2$, where $\p_1$ and $\p_2$ are isomorphic to $\p$. But then from \cite[Theorem~2]{CN2019} we find that any GO metric is normal, hence naturally reductive.

The last remaining case to consider is when the decomposition~\eqref{eq:nirr} has the form $\m = \p \oplus \s \oplus \t$, where $\t$ is a trivial module. Note that $\p$ is of real type and so by Lemma~\ref{l:trivial} we have $[\t, \p]=0, \; [\t, \s] \subset \s$ and $\t \subset \sp(1)$. From \cite[Theorem~2]{CN2019} we also find that $\t \ne 0$. Moreover, the irreducible decomposition of the tensor product $\p \otimes \s$ does not contain $\p$, and so $[\p, \s] \subset \s$ and $[\p, \p] \perp \s$. Also, the irreducible decomposition of the tensor product $\s \otimes \s$ does not contain $\s$. It follows that $(\g, \h'=\h \oplus \p \oplus \t)$ is a symmetric pair, with $\t$ a nonzero ideal in $\h'$. The dimension of the corresponding symmetric space is $4n$, and we also know that $\dim \h' \in \{4n^2, 4n^2+2\}, \; \rk \h' \ge n+1$ and that $\h'$ contains a nonzero ideal $\t \subset \sp(1)$. From the classification in \cite{Hel} we find that the only possibility is $(\g,\h')=(\su(2n+1), \su(2n) \oplus \br)$. Then $\s$ is the orthogonal complement to $\su(2n)$ in $\su(2n+1)$ and $\p$ is the orthogonal complement to $\h$ in $\su(2n)$ (note that the only way, up to automorphism, to realise $\sp(n)$ as a subalgebra of $\su(2n)$ is the one corresponding to the symmetric pair $(\su(2n), \sp(n))$).

We now consider the GO condition. Suppose $\s, \p$ and $\t$ lie in the eigenspaces of $A$ with the eigenvalues $\a_1, \a_2$ and $\a_3$ respectively (some of them, but not all three, may be equal). Then \eqref{eq:GO} is equivalent to the fact that for any $X \in \s, \, Y \in \p$ and $T \in \t$, there exists $Z \in \h$ such that
\begin{equation} \label{eq:spngo}
  [Z, Y] = 0, \qquad [Z + \rho Y + \tau T, X]=0,
\end{equation}
where $\rho= 1-\a_2 \a_1^{-1}, \, \tau = 1 - \a_3 \a_1^{-1}$. We can now identify $\s$ with $\bc^{2n}$ in such a way that the action of $\ad_\t$ is the multiplication by a real multiple of $\ic$ and choose a unitary basis in such a way that $\su(2n)$ is the space of complex matrices of the form $\left(\begin{smallmatrix} K_1 & L \\ -L^* & K_2 \end{smallmatrix}\right)$, where $K_1$ and $K_2$ are skew-Hermitian with $\Tr K_1 + \Tr K_2 =0$ and $L$ is an arbitrary $n \times n$ complex matrix, and $\sp(n) \subset \su(2n)$ is the space of complex matrices of the form $\left(\begin{smallmatrix} K & S \\ -\overline{S} & \overline{K} \end{smallmatrix}\right)$, where $K$ is skew-Hermitian and $S$ is an $n \times n$ symmetric complex matrix. Then $\p$ is the space of matrices of the form $\left(\begin{smallmatrix} K & L \\ \overline{L} & -\overline{K} \end{smallmatrix}\right)$, where $K$ is skew-Hermitian with $\Tr K = 0$ and $L \in \so(n, \bc)$. Choose a generic $Y \in \p$. We can specify the basis further (conjugate by an element of $\Sp(n)$) in such a way that the maximal abelian subalgebra of $\p$ containing $Y$ is the space $\left(\begin{smallmatrix} \ic D & 0 \\ 0 & \ic D \end{smallmatrix}\right)$ where $D$ is a real diagonal matrix with $\Tr D = 0$. Taking $Y$ of this form with $D=\diag(d_1, \dots, d_n)$ such that $d_i \in \br$ are nonzero and pairwise distinct we obtain from the first equation of \eqref{eq:spngo} that $Z=\left(\begin{smallmatrix} K & S \\ -\overline{S} & \overline{K} \end{smallmatrix}\right)$, where $K=\diag(\ic x_1, \dots, \ic x_n), \; S = \diag(z_1, \dots, z_n), \; x_j \in \br, \, z_j \in \bc$. We also have $\ad_T= \ic t I_{2n}, \; t \in \br$. Then the second equation of \eqref{eq:spngo} splits into $n$ pairs of equations of the form
\begin{equation*}
  \ic (x_j + \rho d_j + \tau t)w_j + z_j w_{n+j}=0, \qquad -\overline{z_j} w_j + \ic (-x_j + \rho d_j + \tau t) w_{n+j}=0, \qquad j=1, \dots, n,
\end{equation*}
where $w_j, w_{n+j} \in \bc$ are the corresponding coordinates of $X \in \s$ relative to the chosen basis. Now if $w_j = w_{n+j} = 0$ we can choose $x_j$ and $z_j$ arbitrarily. If not, a direct calculation gives $x_j=(|w_{n+j}|^2-|w_j|^2)(|w_{n+j}|^2+|w_j|^2)^{-1} (\rho d_j + \tau t), \; z_j = -2 \ic w_j \overline{w_{n+j}} (|w_{n+j}|^2+|w_j|^2)^{-1} (\rho d_j + \tau t)$. As these expressions are continuous in $d_j$, the equations \eqref{eq:spngo} have a solution $Z \in \h$ for all $Y \in \p, \, X \in \s$ and $T \in \t$. Hence the metric so defined is GO which gives the space $\SU(2n+1)/\Sp(n), \; n \ge 2$, in Theorem~\ref{th:simple}\eqref{it:thirred}\eqref{it:susp}.

\subsection{Exceptional groups}
\label{ss:ex}

\subsubsection{\underline{$\G_2$}} 
\label{sss:g2}
From Table~\ref{t:tiny}, there is only one tiny module, the space of imaginary octonions $\Oc':=\Oc \cap 1^\perp$ which is the space of the defining representation for $\G_2$, and so the decomposition \eqref{eq:nirr} takes the form $\m = q \Oc' \oplus \t$, where $\t$ is a trivial module. As any three non-associating imaginary octonions generate the whole algebra of octonions $\Oc$, the sum of three copies of $\Oc'$ is a large module (one can check that the sum of two is still a small module, with the principal stationary subgroup $\Sp(1)$), and so by Lemma~\ref{l:twocoml} and Lemma~\ref{l:trivial}\eqref{it:onenont} we have $q \in \{2, 3\}$. As the module $\Oc'$ is of real type, Lemma~\ref{l:trivial}\eqref{it:actionofVa} and \eqref{it:noideal} imply that no $\m_i$ can be trivial (and hence not all the modules $\Oc'$ lie in a single $\m_i$) and that $\m_i$ may contain a nonzero trivial submodule only if it also contains at least two copies of $\Oc'$. If $q=2$ these conditions imply that $\m=\m_1 \oplus \m_2$, where both $\m_1$ and $\m_2$ are isomorphic to $\Oc'$, so that $\m$ is the sum of two irreducible submodules. As $\g$ is simple, by \cite[Theorem~2]{CN2019} we obtain $M= \Spin(8)/\G_2$, the space \eqref{it:8g2} in Theorem~\ref{th:simple}\eqref{it:thirred}.

Suppose $q=3$. If the three modules $\Oc'$ lie in three different eigenspaces of $A$, then similar to the above, we get $\m=3\Oc'$. Then $\dim \g = 35$, and so $\g=\su(6)$. But $\g_2$ is not a subalgebra of $\su(6)$ (as the complexification of $\g_2$ has no nontrivial representation on $\bc^6$), a contradiction. So up to relabelling we have $\m=\m_1 \oplus \m_2$, where $\m_1 = \n_1 \oplus \n_2 \oplus \t_1, \; \m_2 = \n_3$, where the modules $\n_1, \n_2$ and $\n_3$ are isomorphic to $\Oc'$ and $\t_1$ is a trivial module. If $\t_1=0$, we get a contradiction by the dimension count as above. Furthermore, by Lemma~\ref{l:trivial}\eqref{it:actionofVa} we get $[\t_1, \n_3]=0$, and so by Lemma~\ref{l:trivial}\eqref{it:actionofVa} and \eqref{it:noideal}, for every nonzero $T \in \t_1$, the restriction of $\ad_T$ to $\n_1 \oplus \n_2$ is nonzero and commutes with $\ad_\h$, and so we obtain that $\t_1 = \br T$, where $(\ad_T)_{\n_1 \oplus \n_2}$ is given by the matrix $\left(\begin{smallmatrix} 0 & I_7 \\ -I_7 & 0 \end{smallmatrix}\right)$, relative to bases for $\n_1, \n_2$ which correspond via isomorphism. Then $\g$ is a simple algebra of dimension $36$ and $\h'=\h \oplus \t_1$ is its subalgebra; moreover, the modules $\n_1 \oplus \n_2$ and $\n_3$ are irreducible for $\h'$ and the metric on $(\n_1 \oplus \n_2) \oplus \n_3$ obtained by the restriction of $A$ is GO. Then the pair $(\g, \h')$ corresponds to \cite[Theorem~2(2)]{CN2019}, and so we get the space $M=\SO(9)/\G_2$ in Theorem~\ref{th:simple}\eqref{it:thirred}\eqref{it:9g2}. To see that the metric is GO we note that $[\t_1, \n_1]=\n_2$ and $[\t_1, \n_2]=\n_1$. By Lemma~\ref{l:brackets}\eqref{it:alphabeta} we have $[\n_3, \n_i] \subset \n_3 \oplus \n_i$ for $i=1,2$, and so acting on both sides by $\t_1$ we get $[\n_3, \n_i] \subset \n_i$ for $i=1,2$, as $[\t_1, \n_3]=0$. Now for $X=N_1+N_2+T'+N_3$, where $N_i \in \n_i$ for $i=1,2,3$ and $T' \in \t_1$, equation~\eqref{eq:GO} gives $\a_1 [Z, N_1] + \a_1 [Z, N_2] + \a_2 [Z, N_3] + (\a_1-\a_2) [N_3, N_1] + (\a_1-\a_2) [N_3, N_2]=0$ and so projecting to the modules $\n_i, \, i=1,2,3$, we get
\begin{equation} \label{eq:so9g2}
  [Z,N_3]=0, \qquad [Z, N_i] = (\a_1^{-1}\a_2-1) [N_3, N_i], \; \text{for } i=1,2.
\end{equation}
If $N_3=0$ we can take $Z=0$. Suppose $N_3 \ne 0$. Identify all three modules $\n_1, \n_2, \n_3$ with a copy of $\Oc'$ via isomorphisms and denote $x_i, \, i=1,2,3$, the image of $N_i$ under this identification. The stationary subalgebra of $x_3$ is $\su(3)$ acting by the standard representation on $\bc^3 = \Oc' \cap x_3^\perp$ (see Table~\ref{t:tiny}; the stationary subalgebras of all nonzero elements are principal). Moreover, the action of $\ad_{N_3}$ on $\Oc'$ commutes with the action of $\su(3)$ and so it is trivial on $x_3$ and is a multiplication by $\mu \ic$ on $\bc^3$ for some $\mu \in \br$. To satisfy equation~\eqref{eq:so9g2} it is sufficient, for any $x_1, x_2 \in \bc^3$, to find $M \in \su(3)$ such that $Mx_1 = \rho \ic x_1, \; Mx_2 = \rho \ic x_2$, where $\rho=\mu(\a_1^{-1}\a_2-1) \in \br$. We can assume that $x_1$ and $x_2$ are be complex orthonormal. Then, relative to a unitary basis $\{x_1,x_2,y\}$ for $\bc^3$ we can take $M=\diag(\rho \ic, \rho \ic, -2\rho \ic\}$.


\subsubsection{\underline{$\F_4$}}
\label{sss:f4}
From Table~\ref{t:tiny}, there is only one tiny module, of dimension $26$, which comes from the $s$-representation for the symmetric space $Q=\E_6/\F_4$. The sum of three such modules is a large module (the sum of two is still small, by the dimension count $52=\dim \F_4 < 2 \dim \Spin(8) = 56$). To see that, we note that the module $\n$ of dimension $26$ can be viewed as the space of $3 \times 3$ traceless, octonion, Hermitian matrices. The algebra $\h=\mathfrak{f}_4$ can be represented as the direct sum of two subspaces, the Lie subalgebra $\mathfrak{g}_2$ whose elements act on a matrix $N \in \n$ componentwise, and the subspace $\so(3, \Oc)$ of the $3 \times 3$ traceless, octonion, anti-Hermitian matrices whose elements act on a matrix $N \in \n$ as a matrix commutator \cite[Theorem~5]{Baez}. Now take $N_1 \in \n$ to be diagonal (then it is real), with pairwise non-equal entries. The corresponding stationary subalgebra of $\h$ is the direct sum of $\mathfrak{g}_2$ and the subspace $\{\diag(a_1, a_2, a_3) \, | \, a_1, a_2, a_3 \in \Oc', \, a_1+a_2+a_3=0\} \subset \so(3, \Oc)$ (note that the stationary subalgebra is $\so(8)$, as per Table~\ref{t:tiny}). Take $N_2, N_3 \in \n$ defined by
\begin{equation*}
  N_2 = \left(
          \begin{array}{ccc}
            0 & 1 & 1 \\
            1 & 0 & 1 \\
            1 & 1 & 0 \\
          \end{array}
        \right), \qquad
  N_3 = \left(
          \begin{array}{ccc}
            0 & x & y \\
            x^* & 0 & z \\
            y^* & z^* & 0 \\
          \end{array}
        \right),
\end{equation*}
where the octonions $x, y, z$ are non-associating. Note that the action of $\mathfrak{g}_2$ on $N_2$ is trivial as all the entries of $N_2$ are real, and that a matrix $Q= \diag(a_1, a_2, a_3)$ with $a_1, a_2, a_3 \in \Oc', \, a_1+a_2+a_3=0$, commutes with $N_2$ only when $Q = 0$. It follows that the stationary subalgebra of the pair of matrices $(N_1, N_2)$ is the subalgebra $\mathfrak{g}_2 \subset \h$ acting on the matrices from $\n$ componentwise. But as the entries $x, y, z$ of the matrix $N_3$ generate the whole algebra of octonions $\Oc$, the only element of $\mathfrak{g}_2$ which maps all of them to zero is zero. It follows that the stationary subalgebra of the triple $(N_1, N_2, N_3)$ is trivial.

Now by the arguments similar to those in \eqref{sss:g2} we get that the decomposition \eqref{eq:nirr} takes the form $\m = q \n \oplus \t$, where $\t$ is a trivial module and $q = 2, 3$. As the module $\n$ is of real type, Lemma~\ref{l:trivial}\eqref{it:actionofVa} and \eqref{it:noideal} imply that no $\m_i$ can be trivial and that $\m_i$ may contain a nonzero trivial submodule only if it also contains at least two copies of $\n$. Then for $q=2$, the only possibility is $\m=\m_1 \oplus \m_2$ and $\m_1=\n_1, \, \m_2 = \n_2$, where $\n_1, \n_2$ are isomorphic to $\n$, which is not possible by \cite[Theorem~2]{CN2019}. Suppose $q=3$. If the modules $\n_1, \n_2, \n_3$ isomorphic to $\n$ lie in three different eigenspaces of $A$, then $\m=\n_1 \oplus \n_2 \oplus \n_3$, and so $\dim \g = 130$, a contradiction, as there is no simple algebra of this dimension. Otherwise, up to relabelling, we have $\m=\m_1 \oplus \m_2$, where $\m_1 = \n_1 \oplus \n_2 \oplus \t_1, \; \m_2 = \n_3$, where $\t_1$ is a trivial module. As $\n$ is of real type, by Lemma~\ref{l:trivial}\eqref{it:actionofVa} we get $[\t_1, \n_3]=0$, and so by Lemma~\ref{l:trivial}\eqref{it:actionofVa} and \eqref{it:noideal} $\t_1 \subset \so(2)$ acting on $\n_1 \oplus \n_2$. This is a contradiction, as there is no simple Lie algebra whose dimension lies in $\{130, 131\}$. It follows that in the case $H=\F_4$, any GO metric is naturally reductive.

\subsubsection{\underline{$\E_6$}}
\label{sss:e6}
From Table~\ref{t:tiny}, there is only one tiny module $\n$, of dimension $54$, which comes from the $s$-representation for the symmetric space $Q=\E_7/\E_6 \SO(2)$. The sum of three such modules is a large module. Indeed, the symmetric space $Q$ has rank $3$, with the restricted root system of type $\mathrm{C}_3$; there are three roots of multiplicity $1$ and six roots of multiplicity $8$ \cite[Table~1]{T2}. The stationary subalgebra of a regular element of a maximal abelian subalgebra $\ag \subset \q=T_oQ$ is $\so(8) \subset \eg_6$, and by \cite[Lemma~2.25(a)]{N}, it acts as the standard representation of $\so(8)$ on each of the six $8$-dimensional root spaces (and acts trivially on the three $1$-dimensional root spaces). If we choose a generic $6$-tuple of vectors, one in each of the six $8$-dimensional root spaces, the stationary subalgebra of their sum will be $\so(2) \subset \so(8) \subset \eg_6$, but if we choose two such $6$-tuples, the stationary subalgebra will be trivial. We now argue as in the previous cases. The decomposition \eqref{eq:nirr} takes the form $\m= q \n \oplus \t$, where $\t$ is a trivial module and $\q \in \{2, 3\}$. In the case $q=2$, as $\n$ is of complex type, the centraliser of $\ad_\h$ in $\so(\n_1 \oplus \n_2)$ (where $\n_1, \n_2$ are isomorphic copies of $\n$) is $\ug(2)$, and so by Lemma~\ref{l:trivial}\eqref{it:actionofVa}, \eqref{it:noideal} and \eqref{it:realS}, the maximal dimension of $\t$ is $4$. Then $186 \le \dim \g \le 190$. The only simple Lie algebra $\g$ whose dimension satisfies this inequality is $\so(20)$, but $\eg_6$ is not a subalgebra of $\so(20)$ as it has no nontrivial real representation on $\br^{20}$. Suppose $q=3$. If all three submodules $\n_1, \n_2$ and $\n_3$ isomorphic to $\n$ lie in different $\m_i$, then by Lemma~\ref{l:trivial}\eqref{it:actionofVa}, each of them is $\ad(\t)$-invariant, and so from Lemma~\ref{l:trivial}\eqref{it:noideal} and \eqref{it:realS}, we obtain $\dim \t \le 3$. If (up to relabelling) $\m_1 \supset \n_1 \oplus \n_2$ and $\m_2 \supset \n_3$ we get $\ad_T \in \ug(2) \oplus \ug(1)$, for all $T \in \t$ and so $\dim \t \le 5$. Finally, if $\m_1 = \n_1 \oplus \n_2 \oplus \n_3 \oplus \t_1$, where $\t_1$ is trivial, then all the other modules $\m_i, \; 1 < i \le m$, are trivial. By \cite[Theorem~2]{CN2019} we can assume that $m \ge 3$, and then by Lemma~\ref{l:trivial}, $\t=\oplus_{i=1}^m \t_i$ is a subalgebra of $\ug(3)$, with $\t_i$ being commuting ideals and $\t_2, \t_3 \ne 0$. It follows that $\dim \t \le 5$. Therefore in all three cases, we have $240 \le \dim \g \le 245$, but there is no simple algebra whose dimension satisfies this inequality. It follows that in the case $H=\E_6$, any GO metric is naturally reductive.

\subsubsection{\underline{$\E_7$}}
\label{sss:e7}
From Table~\ref{t:tiny}, there is only one tiny module $\n$, of dimension $112$, which comes from the $s$-representation for the symmetric space $Q=\E_8/\E_7 \SU(2)$. We show that the sum of two such modules is a large module. The symmetric space $Q$ has rank $4$, with the restricted root system of type $\F_4$; there are $12$ roots of multiplicity $1$ and $12$ roots of multiplicity $8$ \cite[Table~1]{T2}. The stationary subalgebra of a regular element of a maximal abelian subalgebra $\ag \subset \q = T_oQ$ is $\so(8) \subset \eg_7$, and by \cite[Lemma~2.25(a)]{N}, it acts as the standard representation of $\so(8)$ on each of the twelve $8$-dimensional root spaces. If we choose a generic $12$-tuple of vectors, one in each of the twelve $8$-dimensional root spaces, the stationary subalgebra will be trivial. We can therefore assume that the decomposition \eqref{eq:nirr} takes the form $\m=2\n \oplus \t$, where $\t$ is a trivial module. But now from Lemma~\ref{l:trivial} we get $\t \subset \sp(2)$. It follows that $357 \le \dim \g \le 367$. The only simple Lie algebra $\g$ whose dimension lies in this range is $\su(19)$, but $\eg_7$ cannot be its subalgebra as $\eg_7$ has no nontrivial real representation on $\br^{38}$. So in the case $H=\E_7$, any GO metric must be naturally reductive.

This completes the proof of Theorem~\ref{th:simple}.


\begin{thebibliography}{99}

\bibitem{AA}
D.~V.~Alekseevsky, A.~Arvanitoyeorgos, \emph{Riemannian flag manifolds with homogeneous geodesics},
Trans. Amer. Math. Soc. \textbf{359} (2007), 3769-3789.

\bibitem{AN}
D.~V.~Alekseevsky, Y.~G.~Nikonorov, \emph{Compact Riemannian manifolds with homogeneous geodesics}, SIGMA: Symmetry Integrability Geom. Methods Appl. \textbf{093} (2009), 16 pp.

\bibitem{Baez}
J.~Baez, \emph{The octonions}, Bull. Amer. Math. Soc. (N.S.) \textbf{39} (2002), 145-205.

\bibitem{BKV}
J.~Berndt, O.~Kowalski, L.~Vanhecke, \emph{Geodesics in weakly symmetric spaces}, Ann. Global Anal. Geom. \textbf{15} (1997), 153-156.

\bibitem{BN1}
V.~N.~Berestovskii, Y.~G.~Nikonorov, \emph{On $\delta$-homogeneous Riemannian manifolds}, Differential Geom. Appl. \textbf{26} (2008), 514-535.

\bibitem{BN2}
V.~N.~Berestovskii, Y.~G.~Nikonorov, \emph{Clifford-Wolf homogeneous Riemannian manifolds}, J. Differential Geom. \textbf{82} (2009), 467-500.

\bibitem{BN3}
V.~N.~Berestovskii, Y.~G.~Nikonorov, \emph{Riemannian manifolds and homogeneous geodesics}.
Springer Monographs in Mathematics. Springer, Cham, 2020. 

\bibitem{BN4}
V.~N.~Berestovskii, Y.~G.~Nikonorov, \emph{Finite homogeneous metric spaces}, Sib. Math. J. \textbf{60} (2019), 757-773.

\bibitem{CN2019}
Z.~Chen, Y.~Nikonorov, \emph{Geodesic orbit Riemannian spaces with two isotropy summands. I}, Geom. Dedicata, \textbf{203} (2019), 163-178.

\bibitem{DZ}
J.~E.~D'Atri, W.~Ziller, \emph{Naturally Reductive Metrics and Einstein Metrics on Compact Lie Groups}, Memoirs Amer. Math. Soc. \textbf{19} (1979), no. 215.

\bibitem{DK}
W.~Dickinson, M.~Kerr, \emph{The geometry of compact homogeneous spaces with two isotropy summands}, Ann. Global Anal. Geom. \textbf{34} (2008), 329-350.

\bibitem{Du1}
Z.~Du\v{s}ek, \emph{Homogeneous geodesics and g.o. manifolds}, Note Mat. \textbf{38} (2018), 1-15.

\bibitem{Du2}
Z.~Du\v{s}ek, \emph{Geodesic graphs in Randers g.o. spaces}, Comment. Math. Univ. Carolin. \textbf{61} (2020), 195-211.

\bibitem{Go}
C.~S.~Gordon, \emph{Homogeneous Riemannian manifolds whose geodesics are orbits}, in: Topics in Geometry: In Memory of Joseph D'Atri, Progress in Nonlinear Differential Equations, Vol. 20, Birkhäuser-Verlag, Boston, Basel, Berlin, 1996, 155-174.

\bibitem{Goz}
F.~J.~Gozzi, \emph{Representations of compact Lie groups of low cohomogeneity}, S\~{a}o Paulo J. Math. Sci. (2018), doi: 10.1007/s40863-018-0108-x.

\bibitem{He}
C.~He, \emph{Cohomogeneity one manifolds with a small family of invariant metrics}, Geom. Dedicata 157 (2012), 41-90.

\bibitem{Hel}
Helgason S.
\emph{Differential geometry, Lie groups, and symmetric spaces}.
Pure and Applied Mathematics, 80. Academic Press, Inc. New York -- London, 1978.

\bibitem{HH}
W.~C.~Hsiang, W.~Y.~Hsiang, \emph{Differentiable actions of compact connected classical group: II}, Ann. of Math. \textbf{92} (1970), 189-223.

\bibitem{K56}
B.~Kostant, \emph{On differential geometry and homogeneous spaces. II},
Proc. Natl. Acad. Sci. USA \textbf{42} (1956), 354-357.

\bibitem{K}
B.~Kostant, \emph{On holonomy and homogeneous spaces}, Nagoya Math. J. \textbf{12} (1957), 31-54.

\bibitem{KN}
O.~Kowalski, S.~\v{Z}.~Nik\v{c}evi\'{c}, \emph{On geodesic graphs of Riemannian g.o. spaces}, Arch. Math. \textbf{73} (1999), 223-234.

\bibitem{KV}
O.~Kowalski, L.~Vanhecke, \emph{Riemannian manifolds with homogeneous geodesics}, Boll. Un. Math. Ital. B(7) \textbf{5} (1991), 189-246.

\bibitem{Lau}
E.A.~Lauret, J.~Lauret, \emph{The stability of standard homogeneous Einstein manifolds},
preprint, 2021, arXiv: 2112.08469, \url{https://arxiv.org/abs/2112.08469}.

\bibitem{LL}
G.~Leger, E.~Luks, \emph{Generalized derivations of Lie algebras}, J. Algebra, \textbf{228} (2000) 165-203.

\bibitem{Mal}
A.I.~Malcev, \emph{On semi-simple subgroups of Lie groups}, Am. Math. Soc. Transl. \textbf{33} (1950), 43.

\bibitem{N}
T.~Nagano, \emph{The involutions of symmetric spaces II}, Tokyo J. Math. \textbf{15} (1992), 39-82.

\bibitem{NNLO}
Y.~Nikolayevsky, Yu.~Nikonorov, \emph{On invariant Riemannian metrics on Ledger-Obata spaces}, Manuscripta Math. \textbf{158} (2019), 353-370.

\bibitem{Nik2017}
Yu.~G.~Nikonorov, \emph{On the structure of geodesic orbit Riemannian spaces}, Ann. Glob. Anal. Geom. \textbf{52} (2017), 289-311.

\bibitem{Nsp}
Yu.~G.~Nikonorov, \emph{Geodesic orbit Riemannian metrics on spheres}, Vladikavkaz. Mat. Zh. \textbf{15} (2013), 67-76.

\bibitem{S18}
N.~P.~Souris, \emph{Geodesic orbit metrics in compact homogeneous manifolds with equivalent isotropy submodules},
Transform. Groups. \textbf{23(4)} (2018), 1149-1165.

\bibitem{S20}
N.~P.~Souris, \emph{On a class of geodesic orbit spaces with abelian isotropy subgroup}, Manuscripta Math. \textbf{166} (2021), 101-129. 

\bibitem{Sz}
J.~Szenthe, \emph{Sur la connexion naturelle \`{a} torsion nulle}, Acta Sci. Math. (Szeged), \textbf{38} (1976), 383-398.

\bibitem{T1}
H.~Tamaru, \emph{Riemannian g.o. spaces fibered over irreducible symmetric spaces}, Osaka J. Math. \textbf{36} (1999), 835-851.

\bibitem{T2}
H.~Tamaru, \emph{The local orbit types of symmetric spaces under the actions of the isotropy subgroups}, Diff. Geom. Appl. \textbf{11} (1999), 29-38.

\bibitem{Wo}
J.~A.~Wolf, \emph{Harmonic Analysis on Commutative Spaces}, Mathematical Surveys and Monographs, Vol. 142, American Mathematical Society, Providence, RI, 2007.

\bibitem{Yak}
O.~S.~Yakimova, \emph{Weakly symmetric Riemannian manifolds with a reductive isometry group}, Sb. Math. \textbf{195} (2004), 599-614. 

\bibitem{XDY}
M.~Xu, S.~Deng, Z.~Yan, \emph{Geodesic orbit Finsler metrics on Euclidean spaces}, preprint, 2018, arXiv: 1807.02976, \url{https://arxiv.org/abs/1807.02976}.

\bibitem{YD}
Z.~Yan, S.~Deng, \emph{Finsler spaces whose geodesics are orbits}, Diff. Geom. Appl. \textbf{36} (2014), 1-23.

\bibitem{Zil96}
W.~Ziller, \emph{Weakly symmetric spaces}. In:
Topics in geometry. Progr. Nonlinear Differential Equations Appl.\textbf{20}, 355-368. Birkh{\"a}user, Boston, Boston, MA (1996).
\end{thebibliography}
\end{document}